\documentclass[final]{siamltex}
\usepackage{amssymb}
\usepackage{amsmath}
\usepackage{amscd}
\usepackage{latexsym}
\usepackage{cite}
\usepackage{showkeys}
\usepackage{ifthen}
\usepackage[dvips,bookmarks=true]{hyperref}

\providecommand{\Bdsymb}{\Gamma}                   % Tag for boundary operators
\providecommand{\Div}{\operatorname{div}}          % Divergence
\providecommand{\bDiv}{{\Div}_{\Bdsymb}}           % Surface divergence
\providecommand{\curl}{\operatorname{{\bf curl}}}  % Standard curl operator
\providecommand{\bcurl}{{\curl}_{\Bdsymb}}         % vectorial curl on the surface
\providecommand{\grad}{\operatorname{{\bf grad}}}       % standard gradient
\providecommand{\Dim}{\operatorname{dim}}            % dimension
\providecommand{\dim}{\Dim}
\providecommand*{\Kern}[1]{{\rm Ker}({#1})}                  % Kernel
\newcommand{\Va}{{\mathbf{a}}}
\newcommand{\Vb}{{\mathbf{b}}}

\newcommand{\Vh}{{\mathbf{h}}}

\newcommand{\Vk}{{\mathbf{k}}}

\newcommand{\Vn}{{\mathbf{n}}}

\newcommand{\Vp}{{\mathbf{p}}}
\newcommand{\Vq}{{\mathbf{q}}}

\newcommand{\Vu}{{\mathbf{u}}}
\newcommand{\Vv}{{\mathbf{v}}}
\newcommand{\Vw}{{\mathbf{w}}}

\newcommand{\Vy}{{\mathbf{y}}}
\newcommand{\Vz}{{\mathbf{z}}}

\providecommand{\Ba}{{\boldsymbol{a}}}

\providecommand{\Bn}{{\boldsymbol{n}}}

\providecommand{\Bt}{{\boldsymbol{t}}}

\providecommand{\Bx}{{\boldsymbol{x}}}
\providecommand{\By}{{\boldsymbol{y}}}
\providecommand{\Bz}{{\boldsymbol{z}}}

\providecommand{\BH}{{\boldsymbol{H}}}

\providecommand{\BL}{{\boldsymbol{L}}}

\newcommand{\epsilonbf}{\boldsymbol{\epsilon}}

\newcommand{\mubf}{\boldsymbol{\mu}}

\newcommand{\xibf}{\boldsymbol{\xi}}

\providecommand{\Cf}{{\cal F}}

\providecommand{\Ch}{{\cal H}}

\providecommand{\Cp}{{\cal P}}

\providecommand{\Cw}{{\cal W}}
\providecommand{\Cx}{{\cal X}}
\providecommand{\Cy}{{\cal Y}}

\newcommand{\Chv}{\boldsymbol {\cal H}}

\newcommand{\Cpv}{\boldsymbol {\cal P}}

\providecommand{\bbN}{\mathbb{N}}

\providecommand{\bbR}{\mathbb{R}}

\providecommand*{\N}[1]{\left\|{#1}\right\|} % Double bar norm
\newcommand*{\SN}[1]{\left|{#1}\right|}      % Single bar norm
\newcommand*{\SP}[2]{\left({#1},{#2}\right)} % Inner product

\newcommand*{\Op}[1]{\mathsf{#1}} % Operators
\providecommand*{\Lp}[2][\defaultdomain]{L^{#2}({#1})}
\newcommand*{\Lpv}[2][\defaultdomain]{\BL^{#2}({#1})}
\newcommand*{\NLp}[3][\defaultdomain]{\N{#2}_{\Lp[#1]{#3}}}

\newcommand*{\Ltwo}[1][\defaultdomain]{\Lp[#1]{2}}
\newcommand*{\Ltwov}[1][\defaultdomain]{\Lpv[#1]{2}}
\newcommand*{\NLtwo}[2][\defaultdomain]{\NLp[#1]{#2}{2}}

\newcommand*{\SPLtwo}[3][\defaultdomain]{\SP{#2}{#3}_{\Ltwo[#1]}}
\newcommand*{\Linf}[1][\defaultdomain]{L^{\infty}({#1})}

\newcommand*{\Hm}[2][\defaultdomain]{H^{#2}({#1})}
\newcommand*{\Hmv}[2][\defaultdomain]{\BH^{#2}({#1})}
\newcommand*{\bHm}[3][\defaultdomain]{H_{#3}^{#2}({#1})}

\newcommand*{\NHm}[3][\defaultdomain]{{\N{#2}}_{\Hm[{#1}]{#3}}}

\newcommand*{\SNHm}[3][\defaultdomain]{{\SN{#2}}_{\Hm[{#1}]{#3}}}

\newcommand*{\Hone}[1][\defaultdomain]{\Hm[#1]{1}}
\newcommand*{\Honev}[1][\defaultdomain]{\Hmv[#1]{1}}

\newcommand*{\zbHone}[1][\defaultdomain]{\bHm[#1]{1}{0}}

\newcommand*{\NHone}[2][\defaultdomain]{{\N{#2}}_{\Hone[{#1}]}}

\newcommand*{\SNHone}[2][\defaultdomain]{{\SN{#2}}_{\Hone[{#1}]}}

\newcommand{\hlb}{\frac{1}{2}}
\providecommand*{\Hh}[1][\defaultboundary]{\Hm[#1]{\hlb}}

\newcommand*{\SNHh}[2][\defaultboundary]{\SN{#2}_{\Hh[#1]}}

\newcommand*{\Hdiv}[1][\defaultdomain]{\boldsymbol{H}(\Div,{#1})}
\newcommand*{\bHdiv}[2][\defaultdomain]{\boldsymbol{H}_{#2}(\Div,{#1})}
\newcommand*{\zbHdiv}[1][\defaultdomain]{\bHdiv[#1]{0}}
\newcommand*{\kHdiv}[1][\defaultdomain]{\boldsymbol{H}(\Div0,{#1})}
\newcommand*{\bkHdiv}[2][\defaultdomain]{\boldsymbol{H}_{#2}(\Div0,{#1})}
\newcommand*{\zbkHdiv}[1][\defaultdomain]{\bkHdiv[#1]{0}}
\newcommand*{\NHdiv}[2][\defaultdomain]{\N{#2}_{\Hdiv[#1]}}

\newcommand*{\Hcurl}[1][\defaultdomain]{\boldsymbol{H}(\curl,{#1})}
\newcommand*{\bHcurl}[2][\defaultdomain]{\boldsymbol{H}_{#2}(\curl,{#1})}
\newcommand*{\zbHcurl}[1][\defaultdomain]{\bHcurl[#1]{0}}
\newcommand*{\kHcurl}[1][\defaultdomain]{\boldsymbol{H}(\curl0,{#1})}

\newcommand*{\NHcurl}[2][\defaultdomain]{\N{#2}_{\Hcurl[#1]}}

\title{Discrete Compactness for $p$-Version of Tetrahedral Edge Elements}
\author{ R. Hiptmair\thanks{SAM, ETH Z\"urich, CH-8092 Z\"urich,
    hiptmair\symbol{64}sam.math.ethz.ch}}

\newcommand{\Poinc}{\mathsf{R}}
\newcommand{\mesh}{\mathcal{M}}
\renewcommand{\grad}{\operatorname{\bf grad}}

\newcommand{\Hdive}[1][\defaultdomain]{\boldsymbol{H}(\Div_{\epsilonbf}0,{#1})}
\newcommand{\zbHdive}[1][\defaultdomain]{\boldsymbol{H}_{0}(\Div_{\epsilonbf}0,{#1})}

\newtheorem{remark}{Remark}[section]

\begin{document}

\maketitle

\centerline{\textsf{Report 2008-31, Seminar for Applied Mathematics, ETH Zurich}}

\begin{abstract}
  We consider the first family of $\Hcurl$-conforming Ned\'el\'ec finite elements on
  tetrahedral meshes. Spectral approximation ($p$-version) is achieved by keeping the
  mesh fixed and raising the polynomial degree $p$ uniformly in all mesh cells. We
  prove that the associated subspaces of discretely weakly divergence free piecewise
  polynomial vector fields enjoy a long conjectured discrete compactness property
  as $p\to\infty$. This permits us to conclude asymptotic spectral correctness of
  spectral Galerkin finite element approximations of Maxwell eigenvalue problems.
\end{abstract}



\pagestyle{myheadings}
\thispagestyle{plain}
\markboth{R. Hiptmair}{Discrete compactness}

\section{Introduction}
\label{sec:introduction}

Identifying spectrally correct conforming Galerkin approximations of the Maxwell
eigenvalue problem \cite{COD03}: seek\footnote{We use the customary notations for
  Sobolev spaces like $\Hm{s}$, $\Hcurl$, $\Hdiv$, etc., and write $\kHcurl$,
  $\kHcurl$, etc., for the kernels of differential operators. The reader is referred
  to \cite[\S~I.2]{GIR86} and \cite[Sect.~2.4]{HIP02} for more information.}
$\Vu\in\Hcurl$ and $\omega>0$, such that
\begin{gather}
  \label{eq:101}
  \SPLtwo{\mubf^{-1}\curl\Vu}{\curl\Vv} = \omega^{2}\SPLtwo{\epsilonbf\Vu}{\Vv}\quad
  \forall \Vv\in\Hcurl\;,
\end{gather}
($\epsilonbf,\mubf\in(\Linf)^{3,3}$ uniformly positive definite material tensors) has turned
out to be a highly inspiring challenge in numerical analysis. Obviously,
eigenfunctions of \eqref{eq:101} belong to $\Hcurl\cap\zbHdive$ and the compact
embedding $\Ltwov\hookrightarrow\Hcurl\cap\Hdive$ \cite{JOC97} relates \eqref{eq:101}
to an eigenvalue problem for a compact selfadjoint operator. However, asymptotically
dense families of finite elements in $\Hcurl\cap\Hdive$ are not known in general.

Let us assume that a merely $\Hcurl$-conforming family
$\left(\Cw^{1}_{p}(\mesh)\right)_{p\in\mathbb{N}}$ of finite dimensional trial and
test spaces $\Cw^{1}_{p}(\mesh)\subset\Hcurl$ for \eqref{eq:101} is employed for the
Galerkin discretization of \eqref{eq:101}. The corresponding discrete eigenfunctions
$\Vu_{p}\in \Cw^{1}_{p}(\mesh)$, if they exist, will satisfy
\begin{gather}
  \label{eq:102}
  \Vu_{p} \in \Cx^{1}_{p}(\mesh) := \{\Vw_{p}\in \Cw^{1}_{p}(\mesh):
  \SPLtwo{\epsilonbf\Vw_{p}}{\Vv_{p}} = 0\;\forall \Vv_{p}\in \Kern{\curl}\cap
  \Cw_{p}^{1}(\mesh)\}\;.
\end{gather}
We cannot expect $\Cx_{p}^{1}(\mesh)\subset \Hcurl\cap\zbHdive$ and, thus, a standard
Galerkin approximation of \eqref{eq:101} boils down to an outer approximation of the
eigenvalue problem. Good approximation properties of the finite element space no
longer automatically translate into convergence of eigenvalues and eigenfunctions. As
investigated Caorsi, Fernandes and Rafetto in \cite{CFR00}, an array of other
requirements has to be met by the finite element spaces, the most prominent of which
is the \emph{discrete compactness property} \cite{ANS71}.

\begin{definition}
  \label{def:dc}
  The \emph{discrete compactness property} holds for an asymptotically dense family
  $\left(\Cw^{1}_{p}(\mesh)\right)_{p\in\mathbb{N}}$ of finite dimensional subspaces of
  $\Hcurl$, if any \emph{bounded} sequence in $\Cx_{p}^{1}(\mesh)\subset\Hcurl$ contains a
  subsequence that \emph{converges in $\Ltwov$}. 
\end{definition}

The same notion applies in the case of homogeneous Dirichlet boundary conditions,
when \eqref{eq:101} is considered in $\zbHcurl$. In this case the eigenfunctions 
will belong to $\zbHcurl\cap\Hdive$ and zero tangential trace on $\partial\Omega$
has to be imposed on trial and test functions. 

The discrete compactness property of $\Cw^{1}_{p}(\mesh)$ is key to establishing
spectral correctness and asymptotic optimality of Galerkin approximations of
\eqref{eq:101}, see \cite{CFR00,CFR01,COD03} for details. Small wonder, that
substantial effort has been spent on proving this property for various asymptotically
dense families of $\Hcurl$/$\zbHcurl$-conforming finite elements. For the $h$-version
of Ned\'el\'ec's edge elements Kikuchi \cite{KIK01,KIK89,KIK87} accomplished the
first proof, which was later generalized in \cite{BOF99,DMS00,MOD99}, see
\cite[Sect.~7.3.2]{MON03}, \cite[Sect.~4]{HIP02}, and \cite{COD03} for a survey. Conversely,
spectral edge element schemes in 3D have long defied all attempts to prove their discrete
compactness property, though they perform well for Maxwell eigenvalue problems
\cite{COD03,COL05,RAD00}. Partial success was reported for edge
elements in 2D: In \cite{BDC03} the analysis of the discrete compactness property for
triangular $hp$ finite elements has been tackled, but the proof of the main result
relied on a conjectured $L^{2}$ estimate, which had only been demonstrated numerically.
The first fully rigorous analysis of 2D $hp$ edge elements on rectangles was devised in
\cite{BCD06}.

In \cite[Remark~15]{HIP02} an interpolation estimate was identified as crucial
missing step in the analysis. Since then, two major advances have paved the way for
closing the gap:
\begin{enumerate}
\item In \cite{CMI08} M. Costabel and M. McIntosh discovered a construction of 
  $\Honev$-stable vector potentials by means of a smoothed Poincar\'e mapping.
  This will be reviewed in Sect.~\ref{sec:poincare-lifting} of the present paper.
\item In the breakthrough paper \cite{DEB04} L. Demkowicz and A. Buffa achieved a
  comprehensive analysis of commuting projection based interpolation operators. To
  maintain the article self-contained, their approach will be explained in
  Sect.~\ref{sec:proj-based-interp} and their interpolation error estimates will be
  presented in Sect.~\ref{sec:interp-estim}.
\end{enumerate}
In addition, we exploit the possibility to construct high order versions of
Ned\'el\'ec's first family of edge elements \cite{NED80} by using Cartan's Poincar\'e
map \cite{HIP96b,HIP02,HIP01a,AFW06}, see Sect.~\ref{sec:edge-elements} for
details. Another important tool are stable polynomial preserving extension operators
developed, for example, in \cite{AID03,MUN97a,DGS07,BAS87}. In addition, we heavily
rely on spectral polynomial approximation estimates, see \cite{MUN97a,SAB98,BEM97b}.

Thus, standing on the shoulders of giants and combining all these profound theories
of numerical analysis, this article manages to give the first proof for the discrete
compactness property of the $p$-version for the first family of Ned\'el\'ec's edge
elements on tetrahedral meshes of Lipschitz polyhedra $\Omega$, consult
Sect.~\ref{sec:discrete-compactness} for the proof.

\begin{theorem}
  \label{thm:MAIN}
  The sequence $\left(\Cw^{1}_{p}(\mesh)\right)_{p\in\bbN}$ of trial spaces generated
  by the $p$-version of the first family of Ned\'el\'ec's $\Hcurl$- or
  $\zbHcurl$-conforming finite elements on a fixed tetrahedral mesh $\mesh$ of a
  bounded Lipschitz polyhedron $\Omega\subset\bbR^{3}$ satisfies the discrete compactness
  property.
\end{theorem}

The idea of the proof is to inspect the $\Ltwo$-orthogonal Helmholtz
decomposition \cite[\S~I.3]{GIR86}
\begin{gather}
  \label{eq:134}
  \Vw_{p} = \widetilde{\Vw}_{p} \oplus_{L^{2}} \Vw_{p}^{0}\;,\quad \Vw_{p}^{0}\in
  \Kern{\curl} \;,
\end{gather}
whose so-called solenoidal components $\widetilde{\Vw}_{p}$ belong to 
$\Hcurl\cap\zbkHdiv$. The above mentioned compact embedding guarantees 
the existence of a subsequence of $\left(\widetilde{\Vw}_{p}\right)_{p\in\bbN}$
that converges in $\Ltwov$. Hence, it ``merely'' takes to show 
$\NLtwo{\widetilde{\Vw}_{p}-\Vw_{p}}\to 0$ for $p\to\infty$ in order to
establish discrete compactness. Clever use of projection operators that enjoy
a commuting diagram property, converts this task to a uniform interpolation
estimate. The core of this paper is devoted to this seemingly humble program.

\begin{remark}
  Generalizations of Thm.~\ref{thm:MAIN} to other families of tetrahedral edge
  elements, and corresponding $hp$-finite element schemes are straightforward
  \cite{BCD06}. For the sake of readability, these extensions will not be pursued in
  the present paper.
  
  Since the Poncar\'e map does not fit a tensor product structure, extending the 
  results of this paper to 3D hexahedral edge elements will take some new ideas.
\end{remark}

\section{Poincar\'e lifting}
\label{sec:poincare-lifting}

Let $D\subset \mathbb{R}^{3}$ stand for a bounded domain that is star-shaped
with respect to a subdomain $B\subset D$, that is,
\begin{gather}
  \label{eq:1}
  \forall \Ba\in B,\,\Bx\in D:\quad 
  \{t \Ba + (1-t)\Bx,\; 0<t<1\} \subset D\;.
\end{gather}

\begin{definition}
  \label{def:pl}
  The \emph{Poincar\'e lifting}\footnote{Bold symbols will generally be used to tag
    vector valued functions and spaces of such.}
  $\Poinc_{\Ba}:\boldsymbol{C}^{0}(\overline{\Omega})\mapsto
  \boldsymbol{C}^{0}(\overline{\Omega})$, $\Ba\in B$, is defined as
  \begin{gather}
    \label{eq:pl}
    \Poinc_{\Ba}(\Vu)(\Bx) := \int\nolimits_{0}^{1} t\Vu(\Bx+t(\Bx-\Ba))\,\mathrm{d}t
    \times(\Bx-\Ba)\;,\quad \Bx\in D\;,
  \end{gather}
  where $\times$ designates the cross product of two vectors in $\mathbb{R}^{3}$.
\end{definition}

This is a special case of the generalized path integral formula for differential
forms, which is instrumental in proving the exactness of closed forms on
star-shaped domains, the so-called ``Poincar\'e lemma'', see \cite[Sect.~2.13]{CAR67}.

The linear mapping $\Poinc_{\Ba}$ provides a right inverse of the $\curl$-operator 
on divergence-free vectorfields, see \cite[Prop.~2.1]{GOD03} for the simple proof, and 
\cite[Sect.~2.13]{CAR67} for a general proof based on differential forms. 

\begin{lemma}
  \label{lem:lift}
  If $\Div\Vu=0$, then, for any $\Ba\in B$, $\curl\Poinc_{\Ba}\Vu = \Vu$ for all
  $\Vu\in \boldsymbol{C}^{1}(\overline{\Omega})$.
\end{lemma}

Unfortunately, the mapping $\Poinc_{\Ba}$ cannot be extended to a continuous mapping
$\Ltwov[D]\mapsto \Honev[D]$, \textit{cf.} \cite[Thm.~2.1]{GOD03}. As discovered in
the breakthrough paper \cite{CMI08} based on earlier work of Bogovski\v{i}
\cite{BOG79}, it takes a smoothed version to accomplish this: we introduce the
\emph{smoothed Poincar\'e lifting} \footnote{The dependence of $\Poinc$ on $\Phi$ is
  dropped from the notation.}
\begin{gather}
  \label{eq:2}
  \Poinc(\Vu) := \int\nolimits_{B} \Phi(\Ba)\Poinc_{\Ba}(\Vu)\,\mathrm{d}\Ba\;,
\end{gather}
where 
\begin{gather}
  \label{eq:3}
  \Phi \in C^{\infty}(\mathbb{R}^{3})\;,\quad
  \operatorname{supp}\Phi \subset B\;,\quad
  \int\nolimits_{B}\Phi(\Ba)\,\mathrm{d}\Ba = 1\;.
\end{gather}
The substitution
\begin{gather}
  \label{eq:103}
  \By := \Ba + t(\Bx-\Ba)\quad,\quad
  \tau := \frac{1}{1-t}\;,
\end{gather}
transforms the integral \eqref{eq:3} into
\begin{gather}
  \label{eq:104}
  \begin{aligned}
  \Poinc(\Vu)(\Bx) & = \int\limits_{\mathbb{R}^{3}}\int\limits_{1}^{\infty}
  \tau(1-\tau)\Vu(\By)\times(\Bx-\By)\Phi(\By+\tau(\By-\Bx))\,\mathrm{d}\tau\mathrm{d}\By \\
  & = \int\limits_{\mathbb{R}^{3}} \Vk(\Bx,\By-\Bx)\times \Vu(\By)\,\mathrm{d}\By\;,
  \end{aligned}
\end{gather}
that is, $\Poinc$ is a convolution-type integral operator with kernel
\begin{gather}
  \label{eq:105}
  \begin{aligned}
    \Vk(\Bx,\Bz) & = \int\nolimits_{1}^{\infty}\tau(1+\tau)\Phi(\Bx+\tau
    \Bz)\Bz\,\mathrm{d}\tau \\
    & = \frac{\Bz}{|\Bz|^{2}} \int\nolimits_{1}^{\infty}
    \zeta\Phi(\Bx+\zeta\frac{\Bz}{|\Bz|})\,\mathrm{d}\zeta + \frac{\Bz}{|\Bz|^{3}}
    \int\nolimits_{1}^{\infty}\zeta^{2}\Phi(\Bx+\zeta \frac{\Bz}{|\Bz|})\,\mathrm{d}\zeta\;.
  \end{aligned}
\end{gather}
The kernel can be bounded by $|\Vk(\Bx,\Bz)| \leq K(\Bx)|\Bz|^{-2}$, where $K\in
C^{\infty}(\bbR^{3})$ depends only on $\Phi$ and is locally uniformly bounded. 
As a consequence, \eqref{eq:104} exists as an improper integral. 

The intricate but elementary analysis of \cite[Sect.~3.3]{CMI08} further shows, that
$\Vk$ belongs to the H\"ormander symbol class $S^{-1}_{1,0}(\bbR^{3})$, see
\cite[Ch.~7]{TAY96y}. Invoking the theory of pseudo-differential operators
\cite[Prop.~5.5]{TAY96y} we obtain the following following continuity result, which
is a special case of \cite[Cor.~3.4]{CMI08}

\begin{theorem}
  \label{thm:CMI}
  The mapping $\Poinc$ can be extended to a \emph{continuous} linear operator
  $\Ltwov[D]\mapsto\Honev[D]$, which is still denoted by $\Poinc$. It 
  satisfies
  \begin{gather}
    \label{eq:cmi}
    \curl\Poinc\Vu = \Vu\quad \forall \Vu\in\kHdiv[D]\;.
  \end{gather}
\end{theorem}

The smoothed Poincar\'e lifting shares this continuity property with many other
mappings, see \cite[Sect.~2.4]{HIP02}. Yet, it enjoys another essential feature,
which is immediate from its definition \eqref{eq:pl}: $\Poinc$ maps polynomials
of degree $p$ to other polynomials of degree $\leq p+1$. The next section will
highlight the significance of this observation.

\section{Tetrahedral edge elements}
\label{sec:edge-elements}

In \cite{NED80} Ned\'el\'ec introduced a family of $\Hcurl$-conforming, that is,
tangentially continuous, finite element spaces. On a tetrahedral triangulation
$\mesh$ of $\Omega$, the corresponding finite element spaces of degree $p$ are given
by
\begin{align*}
%  \label{eq:4}
  \Cw^{1}_{p}(\mesh) & 
  := \{\Vv\in\Hcurl:\; {\Vv}_{|T}
  \in \Cw_{p}^{1}(T)\;\forall T\in\mesh\}\;,\\
%  \label{eq:5}
  \Cw^{1}_{p}(T) & 
  := \{\Vv\in \boldsymbol{C}^{\infty}(T):\;\Vv(\Bx) = \Vp(\Bx) + \Vq(\Bx)\times 
  \Bx,\; \Vp,\Vq \in\Cpv_{p}(\bbR^{3}),\;\Bx\in T\}\;.
\end{align*}
We wrote $\mathcal{P}_{p}(\bbR^{3})$ for the space of 3-variate polynomials of total
degree $\leq p$, $p\in\mathbb{N}_{0}$, and the bold symbol $\Cpv_{p}(\bbR^{3})$ for
vectorfields with three components in $\mathcal{P}_{p}(\bbR^{3})$. To emphasize that
polynomials on a tetrahedron $T$ are being considered, we may use the notations
$\Cp_{p}(T)$/$\Cpv_{p}(T)$ instead of $\Cp_{p}(\bbR^{3})$/$\Cpv(\bbR^{3})$. We also
adopt the convention that $\Cp_{p}(\bbR^{3}) = \{0\}$, if $p<0$. Another relevant
polynomial space is
\begin{gather}
  \label{eq:71}
  {\boldsymbol\Cp}_{p}(\Div0,\bbR^{3}) := \{\Vq\in \Cpv_{p}(\bbR^{3}):\;\Div\Vq=0\}\;.
\end{gather}%
Deep insights can be gained by regarding edge elements as discrete 1-forms. This
provides a very elegant construction of higher order edge element spaces and
immediately reveals their relationships with standard Lagrangian finite elements and
$\Hdiv$-conforming face elements (see below). In particular, the Poincar\'e lifting
becomes a powerful tool for building discrete differential forms of high polynomial
degree. This is explored in \cite{HIP96b,HIP01a}, \cite[Sect.~3.4]{HIP02}, and
\cite[Sect.~1.4]{AFW06} in arbitrary dimension, using the calculus of differential
forms. In this article we prefer to stick to the classical calculus of vector
analysis, because we are only concerned with 3D. We hope, that, thus, the
presentation will be more accessible to an audience of numerical analysts. Yet, the
differential forms background has inspired our notations: integer superscripts
label spaces and operators related to differential forms. For instance,
%$\Cw^{%\rnode{SS1}{\scriptstyle 1}\rput(0.5em,-1.5em){\pnode{CCt}}}_{p}(\mesh)$ 
$\Cw^{1}_{p}(\mesh)$ 
can be read as a space of discrete 1-forms.
%\ncline{->}{CCt}{SS1}

According to \cite[Sect.~3]{HIP96b}, for any $T\in\mesh$, $\Ba\in T$, we can obtain
the local space as
\begin{gather}
  \label{eq:6}
  \Cw^{1}_{p}(T) = \Cpv_{p}(\bbR^{3}) + 
  \Poinc_{\Ba}\bigl({\boldsymbol\Cp}_{p}(\Div0,\bbR^{3})\bigr)\;.
\end{gather}
Independence of $\Ba$ is discussed in \cite[Sect.~3]{HIP96b}. The representation
\eqref{eq:6} can be established by dimensional arguments: from the formula
\eqref{eq:pl} for the Poincar\'e lifting we immediately see that $ \Cpv_{p}(\bbR^{3}) +
\Poinc_{\Ba}(\Cpv_{p}(\bbR^{3})) \subset \Cw^{1}_{p}(T)$. In addition, from
\cite[Lemma~4]{NED80} and \cite[Thm.~6, case $l=1$, $n=3$]{HIP96b} we learn that the
dimensions of both spaces agree and are equal to
\begin{gather}
  \label{eq:9}
  \dim \Cw^{1}_{p}(T) = \tfrac{1}{2}(1+p)(3+p)(4+p)\;.
\end{gather}
As a consequence, the two finite dimensional spaces must agree. 

For the remainder of this section, which focuses on local spaces, we single out a
tetrahedron $T\in\mesh$.  On $T$ we can introduce a smoothed Poincar\'e lifting
$\Poinc_{T}$ according to \eqref{eq:2} with $B=T$ and a suitable $\Phi\in
C^{\infty}_{0}(T)$ complying with \eqref{eq:3}. An immediate consequence of
\eqref{eq:6} is that
\begin{gather}
  \label{eq:72}
  \Poinc_{T}\bigl({\boldsymbol\Cp}_{p}(\Div0,\bbR^{3})\bigr) \subset \Cw^{1}_{p}(T)\;.
\end{gather}

We introduce the notation $\Cf_{m}(T)$ for the set of all $m$-dimensional facets of
$T$, $m=0,1,2,3$.  Hence, $\Cf_{0}(T)$ contains the vertices of $T$, $\Cf_{1}(T)$ the
edges, $\Cf_{2}(T)$ the faces, and $\Cf_{3}(T) = \{T\}$. Moreover, for some $F\in
\Cf_{m}(T)$, $m=1,2,3$, $\Cp_{p}(F)$ denotes the space of $m$-variate polynomials of
total degree $\leq p$ in a local coordinate system of the facet $F$, and
$\Cpv_{p}(F)$ will designate corresponding tangential polynomial vectorfields.
Further, we write
\begin{gather}
  \label{eq:12}
  \Cw^{1}_{p}(e) = \Cw^{1}_{p}(T)\cdot \Bt_{e}\;,\quad 
  \Bt_{e}\text{ the unit tangent vector of }e,\;e\in\Cf_{1}(T)\;,\\
  \label{eq:13}
  \Cw^{1}_{p}(f) = \Cw^{1}_{p}(T)\times \Bn_{f}\;,\quad
  \Bn_{f}\text{ the unit normal vector of }f,\;f\in\Cf_{2}(T)\;,
\end{gather}
for the tangential traces of local edge element vectorfields onto edges and faces.
Simple vector analytic manipulations permit us to deduce from \eqref{eq:6} that
\begin{align}
  \label{eq:14}
  \Cw_{p}^{1}(e) & = \Cp_{p}(e)\;,\quad e\in\Cf_{1}(T)\;,\\
  \label{eq:15}
  \Cw_{p}^{1}(f) & = \Cpv_{p}(f) + \Poinc_{\Ba}^{2D}(\Cp_{p}(f))\;,\quad
  \Ba\in f\;,\quad f\in\Cf_{2}(T)\;,
\end{align}
where the projection $\Poinc_{\Ba}^{2D}$ of the Poincar\'e lifting in the plane
reads
\begin{gather}
  \label{eq:16}
  \Poinc_{\Ba}^{2D}(u)(\Bx) := \int\nolimits_{0}^{1} t
  u(\Ba+t(\Bx-\Ba)](\Bx-\Ba)\,\mathrm{d}t\;,
  \quad \Ba\in\bbR^{2}\;.
\end{gather}
It satisfies $\bDiv\Poinc_{\Ba}^{2D}(u) = u$ for all $u\in C^{\infty}(\bbR^{2})$.  We
point out that, along with \eqref{eq:6}, the formulas \eqref{eq:14} and \eqref{eq:15}
are special versions of the general representation formula for discrete 1-forms, see
\cite[Formula~(16)]{HIP96b}. Special facet tangential trace spaces will also be
needed:
\begin{align}
  \label{eq:18}
  \smash{\smash{\overset{\circ}{\Cw}}}_{p}^{1}(e) & := \{u\in \Cw^{1}_{p}(e):\;
  \int\nolimits_{e}u\,\mathrm{d}l = 0 \}\;,\quad
  e\in \Cf_{1}(T)\;,\\[1.5ex]
  \label{eq:136}
  \smash{\overset{\circ}{\Cw}}_{p}^{1}(f) & := 
  \{\Vu\in \Cw^{1}_{p}(f):\; \Vu\cdot\Bn_{e,f} \equiv 0\;\forall e\in
  \Cf_{1}(T),\;e\subset\partial{f}\}\;,\quad f\in\Cf_{2}(T)\;,\\[1.5ex]
  \label{eq:137}
  \smash{\overset{\circ}{\Cw}}_{p}^{1}(T) & := 
  \{\Vu\in \Cw^{1}_{p}(T):\; \Vu\times\Bn_{f} \equiv 0\;\forall f\in
  \Cf_{2}(T)\}\;.
\end{align}%
Here $\Bn_{f}$ represents an exterior face unit normal of $T$, $\Bn_{e,f}$ the in
plane normal of a face w.r.t. an edge $e\subset\partial f$.

According to \cite[Sect.~1.2]{NED80} and \cite[Sect.~4]{HIP96b}, the local degrees of
freedom for $\Cw^{1}_{p}(T)$ are given by the first $p-2$ vectorial moments on the
cells of $\mesh$, the first $p-1$ vectorial moments of the tangential components on
the faces of $\mesh$ and the first $p$ tangential moments along the edges of $T$, see
\eqref{PRNDdof} for concrete formulas. Then the set $\mathrm{dof}_{p}^{1}(T)$ can be
partitioned as
\begin{gather}
  \label{eq:11}
  \mathrm{dof}_{p}^{1}(T) = \bigcup\limits_{e\in\Cf_{1}(T)} \mathrm{ldf}_{p}^{1}(e)\;
  \cup \bigcup\limits_{f\in\Cf_{2}(T)} \mathrm{ldf}_{p}^{1}(f) \;\cup\;
  \mathrm{ldf}_{p}^{1}({T})\;,
\end{gather}
where the functionals in $\mathrm{ldf}_{p}^{1}(e)$, $\mathrm{ldf}_{p}^{1}(f)$, and
$\mathrm{ldf}_{p}^{1}({T})$ are supported on an edge, face,
and $T$, respectively, and read
\begin{gather}
  \label{PRNDdof}
  \renewcommand{\arraystretch}{1.3}
  \begin{array}{rcll}
    \kappa\in \mathrm{ldf}_{p}^{1}(e) & \Rightarrow & \kappa(\Vu) =
    \int\nolimits_{{e}}p {\xibf}\cdot{\Bt}_{e}\,\mathrm{d}l &
    \text{ for }{e}\in\Cf_1({T}),
    \text{ suitable }p\in\Cp_{p}(e)\;, \\
    \kappa\in \mathrm{ldf}_{p}^{1}(f) & \Rightarrow &\kappa(\Vu)=
    \int\nolimits_{{f}}{\Vp}\cdot({\xibf\times\Vn})\,\mathrm{d}S &
    \text{ for }{f}\in\Cf_2({T}),
    \text{ suitable }\Vp\in\Cpv_{p-1}(f)\;, \\
    \kappa\in \mathrm{ldf}_{p}^{1}({T}) & \Rightarrow &\kappa(\Vu):= 
    \int\nolimits_{{T}}{\Vp}\cdot{\xibf}\,\mathrm{d}\Bx&
    \text{for suitable }\Vp\in\Cpv_{p-2}(T)\;.
  \end{array}
\end{gather}
These functionals are unisolvent on $\Cw_{p}^{1}(T)$ and locally fix the tangential
trace of $\Vu\in \Cw^{1}_{p}(T)$. There is a splitting of $\Cw_{p}^{1}(T)$ dual to 
\eqref{eq:11}: Defining
\begin{gather}
  \label{eq:17}
  \Cy^{1}_{p}(F) :=  \{\Vv\in \Cw^{1}(T):\; \kappa(\Vv) = 0\;\forall 
  \kappa\in \mathrm{dof}_{p}^{1}(T)\setminus\mathrm{ldf}_{p}^{1}(F)\}
\end{gather}
for $F\in\Cf_{m}(T)$, $m=1,2,3$, we find the direct sum decomposition
\begin{gather}
  \label{eq:10}
  \Cw_{p}^{1}(T) = \sum\limits_{m=1}^{3} \sum\limits_{F\in\Cf_{m}(T)} \Cy_{p}^{1}(F)\;.
\end{gather}
In addition, note that the tangential trace of $\Vu\in\Cx_{p}^{1}(F)$ vanishes
on all facets $\not= F$, whose dimension is smaller or equal the dimension of
$F$. By the unisolvence of $\mathrm{dof}_{p}^{1}(T)$, there are bijective linear
\emph{extension operators}
\begin{align}
  \label{eq:19}
  \mathsf{E}^{1}_{e,p} & :\Cw^{1}_{p}(e)\mapsto \Cy^{1}_{p}(e)\;,\quad e\in\Cf_{1}(T)\;,\\
  \mathsf{E}^{1}_{f,p} & :\smash{\overset{\circ}{\Cw}}^{1}_{p}(f)\mapsto\Cy^{1}_{p}(f)\;,\quad
  f\in\Cf_{2}(T)\;.
\end{align}

The $\curl$ connects the edge element spaces $\Cw^{1}_{p}(\mesh)$ and the so-called
face element spaces of discrete 2-forms \cite[Sect.~1.3]{NED80}
\begin{align*}
  \Cw^{2}_{p}(\mesh) & := \{\Vv\in\Hdiv:\;{\Vv}_{|T} \in \Cw^{2}_{p}(T)\;\forall T\in\mesh\}\;,\\
  \Cw^{2}_{p}(T) & := \{\Vv\in\boldsymbol{C}^{\infty}(T):\,
  \Vv(\Bx)=\Vp(\Bx)+q(\Bx)\Bx,\;\Vp\in\Cpv_{p}(T),\,q\in\Cp_{p}(T)\}\;.
\end{align*}
An alternative representation of the local face element space is 
\cite[Formula (16) for $l=2$, $n=3$]{HIP96b}
\begin{gather}
  \label{eq:106}
  \Cw^{2}_{p}(T) = \Cpv_{p}(T) + \mathsf{D}_{\Ba}(\Cp_{p}(T))\;,
\end{gather}
where the appropriate version of the Poincar\'e lifting reads
\begin{gather}
  \label{eq:107}
  (\mathsf{D}_{\Ba}u)(\Bx) := \int\nolimits_{0}^{1}t^{2}
  u(\Ba+t(\Bx-\Ba))(\Bx-\Ba)\,\mathrm{d}t
  \;,\quad \Ba\in T\;.
\end{gather}
Like \eqref{eq:6} this is a special incarnation of the general formula (16) 
in \cite{HIP96b}. Again, dimensional arguments based on \cite[Sect.~1.3]{NED80}
and \cite[Thm.~6]{HIP96b} confirm the representation \eqref{eq:107}. We
remark that $\Div\mathsf{D}_{\Ba}u = u$, see \cite[Prop.~1.2]{GOD03}.

The normal trace space of $\Cw^{2}_{p}(T)$ onto a face is
\begin{gather}
  \label{eq:116}
  \Cw^{2}_{p}(f) := \Cw_{p}^{2}(T)\cdot\Bn_{f} = \Cp_{p}(f)\;,\quad f\in\Cf_{2}(T)\;,
\end{gather}
and as relevant space ``with zero trace'' we are going to need
\begin{align}
  \label{eq:117}
  \smash{\overset{\circ}{\Cw}}_{p}^{2}(f) & := \{u\in\Cw^{2}_{p}(f):\int\nolimits_{f} 
  u\,\mathrm{d}S = 0\}\;,\quad f\in\Cf_{2}(T)\;,\\
  \smash{\overset{\circ}{\Cw}}_{p}^{2}(T) & := 
  \{\Vu\in\Cw_{p}^{2}(T):\; \Vu\cdot\Bn_{\partial T}=0\}\;.
\end{align}

The connection between the local spaces $\Cw^{1}_{p}(T)$, $\Cw^{2}_{p}(T)$ and
full polynomial spaces is established through a local discrete DeRham exact
sequence: To elucidate the relationship between differential operators
and various traces onto faces and edges, we also include those in the statement
of the following theorem. There $\Bn_{f}$ stands for an exterior face unit normal
of $T$, $\Bn_{e,f}$ for the in plane normal of a face w.r.t. an edge
$e\subset\partial f$, and $\frac{d}{dl}$ is the differentiation w.r.t. arclength
on an edge. 

\begin{theorem}
  \label{thm:locDR}
  For $f\in\Cf_{2}(T)$, $e\in\Cf_{1}(T)$, $e\subset\partial f$, all the sequences
  in 
  \begin{gather*}
    \begin{CD}
      \mathrm{const} @>{\mathsf{Id}}>> \Cp_{p+1}(T) @>{\grad}>> \Cw^{1}_{p}(T) 
      @>{\curl}>> \Cw^{2}_{p}(T) @>{\Div}>> \Cp_{p}(T) @>{\mathsf{Id}}>> \{0\} \\
      & & @V{{.}_{|f}}VV @V{{.\times\Bn_{f}}_{|f}}VV @VV{{.\cdot\Bn_{f}}_{|f}}V \\
      \mathrm{const} @>{\mathsf{Id}}>>
      \Cp_{p+1}(f) @>{\bcurl}>> \Cw^{1}_{p}(f) @>{\bDiv}>> \Cp_{p}(f)
      @>{\mathsf{Id}}>> \{0\}\\
      & & @V{{.}_{|e}}VV @V{{.\cdot\Bn_{e,f}}_{|e}}VV  \\
      \mathrm{const} @>{\mathsf{Id}}>> \Cp_{p+1}(e) @>{\frac{d}{dl}}>> \Cp_{p}(e) 
      @>{\mathsf{Id}}>> \{0\}
    \end{CD}
  \end{gather*}
  are exact and the diagram commutes. 
\end{theorem}

\begin{proof}
  The assertion about the top exact sequence is an immediate consequence of
  representations \eqref{eq:6} and \eqref{eq:106} and the relationships
  \begin{gather*}
    \curl\Poinc_{\Ba}(\Vu) = \Vu\quad\forall \Vu \in \Cpv_{p}(\Div0,T)\,,\quad
    \Div\mathsf{D}_{\Ba}(u) = u \quad\forall u\in\Cp_{p}(T)\;.
  \end{gather*}
  For further discussions and the proof of the other exact sequence properties see
  \cite[Sect.~5 for $n=3$]{HIP96b}.
\end{proof}

% The degrees of freedom suggest another local construction of $\Ce_{p}(K)$, $K$ a
% tetrahedron (with vertices $\Bp_{i}$, $i=1,\ldots,4$, which was first pursued in
% \cite[Sect.~3.4]{HIP02} and will be explained in the sequel: Writing $\lambda_{i}$
% for the barycentric coordinate functions of $K$, we first recall the formula for the
% local shape functions for the lowest order edge element space. They are attached to
% the edges of $K$ and read for the directed edge $e:=[\Vp_{i},\Vp_{j}]$ 
% \begin{gather}
%   \label{eq:7}
%   \Vb_{e} = \lambda_{i}\grad \lambda_{j} - \lambda_{j}\grad \lambda_{i}\;.
% \end{gather}
% These functions span $\Ce_{0}(K)$. The local space of degree $p$ edge functions
% on $e$ is 
% \begin{gather}
%   \label{eq:8}
%   \Ce_{e,p} := \operatorname{Span}\{\widetilde{p}(\lambda_{i},\lambda_{j})\lambda_{j}\grad
%   \lambda_{i}\}\;,
% \end{gather}
% where $\widetilde{p}$ is a 2-variate \emph{homogeneous} ploynomial of degree $p-1$
% that uniquely determines a function in $\Ce_{e,p}$. 

\section{Projection based interpolation}
\label{sec:proj-based-interp}

The degrees of freedom introduced above define local finite element projectors onto
$\Cw_{p}^{1}(T)$. In conjunction with suitably defined interpolation operators for
degree $p$ Lagrangian finite elements, they possess a very desirable commuting
diagram property \cite[Thm.~13]{HIP96b}, which will be explained below.  However,
they do not enjoy favorable continuity properties with increasing $p$. Thus, L.
Demkowicz \cite{DEB01,DEM04a,DEB04}, taking the cue from the theory of $p$-version
Lagrangian finite elements, invented an alternative in the form of local projection
based interpolation. 

\subsection{Projections, liftings, and extensions}
\label{sec:proj-lift-extens}

Again, consider a single tetrahedron $T\in\mesh$ and fix the polynomial degree
$p\in\bbN$. Following the developments of \cite[Sect.~3.5]{HIP02}, projection based
interpolation requires building blocks in the form of local \emph{orthogonal}
projections $\mathsf{P}_{\ast}^{l}$ and liftings $\mathsf{L}_{\ast}^{l}$\footnote{%
  The parameter $l$ in the notations for the extension operators
  $\mathsf{E}_{\ast}^{l}$, the projections $\mathsf{P}_{\ast}^{l}$, and the liftings
  $\mathsf{L}_{\ast}^{l}$ refers to the degree of the discrete differential form they
  operate on. This is explained more clearly in \cite[Sect.~3.5]{HIP02}.}.  Some
operators will depend on a regularity parameter $0<\epsilon<\frac{1}{2}$, which is
considered fixed below and will be specified in Sect.~\ref{sec:interp-estim}. To
begin with, we define for every $e\in\Cf_{1}(T)$
\begin{gather}
  \label{eq:21}
  \mathsf{P}^{1}_{e,p}:H^{-1+\epsilon}(e) \mapsto \frac{d}{dl}
  \smash{\overset{\circ}{\Cp}}_{p+1}(e) = 
  \smash{\overset{\circ}{\Cw}}^{1}_{p}(e)
\end{gather}
as the $H^{-1+\epsilon}(e)$-orthogonal projection. Here,
$\smash{\overset{\circ}{\Cp}}_{p}(F)$ denotes the space of degree $p$ polynomials on
a facet $F$ that vanish on $\partial F$.

Similarly, for every face $f\in\Cf_{2}(T)$ introduce
\begin{align}
  \label{eq:22}
  \mathsf{P}_{f,p}^{1}: & \Hmv[f]{-\frac{1}{2}+\epsilon}\mapsto 
  \bcurl\smash{\overset{\circ}{\Cp}}_{p+1}(f) = 
  \{\Vv\in \smash{\overset{\circ}{\Cw}}^{1}_{p}(f):\;\bDiv\Vv = 0\}\;,\\
   \label{eq:25}
  \mathsf{P}_{f,p}^{2}: & \Hmv[f]{-\frac{1}{2}+\epsilon}\mapsto
  \bDiv\smash{\overset{\circ}{\Cw}}^{1}_{p}(f)
  = \smash{\overset{\circ}{\Cw}}_{p}^{2}(f)\;, 
\end{align}
as the corresponding $\Hmv[f]{-\frac{1}{2}+\epsilon}$-orthogonal projections. Eventually, let
\begin{align}
  \label{eq:23}
  \mathsf{P}_{T,p}^{1}: & \Ltwov[T] \mapsto 
  \grad \smash{\overset{\circ}{\Cp}}_{p+1}(T) = 
  \{\Vv\in\smash{\overset{\circ}{\Cw}}_{p}^{1}(T):\;\curl\Vv=0\}\;,\\
  \label{eq:24}
  \mathsf{P}_{T,p}^{2}: & \Ltwov[T]\mapsto
  \curl\smash{\overset{\circ}{\Cw}}_{p}^{1}(T)
  = \{\Vv\in\smash{\overset{\circ}{\Cw}}_{p}^{2}(T):\; \Div\Vv=0\}\;,\\
  \label{eq:108}
  \mathsf{P}_{T,p}^{3}: & 
  \Ltwo[T] \mapsto \Div\smash{\overset{\circ}{\Cw}}_{p}^{2}(T) = 
  \{v\in\Cp_{p}(T):\;\int\nolimits_{T}v(\Bx)\,\mathrm{d}\Bx=0\}\;,
\end{align}
stand for the respective $\Ltwo[T]$-orthogonal projections. 

The lifting operators 
\begin{align}
  \label{eq:26}
  \mathsf{L}^{1}_{e,p} : & \smash{\overset{\circ}{\Cw}}^{1}_{p}(e) \mapsto 
  \smash{\overset{\circ}{\Cp}}_{p+1}(e)\;,\quad
  e\in\Cf_{1}(T)\;,\\
  \mathsf{L}^{1}_{f,p} : & \{\Vv\in
  \smash{\overset{\circ}{\Cw}}^{1}_{p}(f):\;\bDiv\Vv = 0\} \mapsto
  \smash{\overset{\circ}{\Cp}}_{p+1}(f)\;,\quad 
  f\in\Cf_{2}(T)\;,\\
  \mathsf{L}^{1}_{T,p} : & \{\Vv\in\smash{\overset{\circ}{\Cw}}_{p}^{1}(T):\;\curl\Vv=0\}
  \mapsto \smash{\overset{\circ}{\Cp}}_{p+1}(T)\;,
\end{align}
are uniquely defined by requiring
\begin{align}
  \label{eq:27}
  & \frac{d}{dl}\mathsf{L}^{1}_{e,p} u = u \quad 
  \forall u\in \smash{\overset{\circ}{\Cw}}^{1}_{p}(e)\;,\\
  & \bcurl\mathsf{L}^{1}_{f,p}\Vu = \Vu \quad\forall \Vu\in 
  \{\smash{\overset{\circ}{\Cw}}^{1}_{p}(f):\;\bDiv\Vv = 0\}\;,\\
  & \grad\mathsf{L}^{1}_{T,p}\Vu = \Vu \quad\forall \Vu\in 
  \{\Vv\in\smash{\overset{\circ}{\Cw}}_{p}^{1}(T):\;\curl\Vv=0\}\;.
\end{align}
Another class of liftings provides right inverses for $\curl$ and $\bDiv$:
Pick a face $f\in\Cf_{2}(T)$, and, without loss of generality, assume the vertex
opposite to the edge $\widetilde{e}$ to coincide with $0$. Then define
\begin{gather}
  \label{eq:28}
  \mathsf{L}^{2}_{f,p}: \left\{
    \begin{array}[c]{ccl}
    \bDiv \smash{\overset{\circ}{\Cw}}^{1}_{p}(f)&\mapsto & \smash{\overset{\circ}{\Cw}}^{1}_{p}(f)\\
    u & \mapsto & \Poinc_{0}^{2D}u -
    \bcurl\mathsf{E}^{0}_{\widetilde{e},p}\mathsf{L}^{1}_{\widetilde{e},p}
    (\Poinc_{0}^{2D}u\cdot\Bn_{\widetilde{e},f})\;.
  \end{array}
\right.
\end{gather}
This is a valid definition, since, by virtue of definition \eqref{eq:16}, the normal
components of $\Poinc_{0}^{2D}u$ will vanish on $\partial
f\setminus\widetilde{e}$. Moreover, $\bDiv\Poinc_{0}^{2D}u = u$ ensures that 
the normal component of $\Poinc_{0}^{2D}u$ has zero average on $\widetilde{e}$. We infer
\begin{multline*}
  {\Bigl(\bcurl\mathsf{E}^{0}_{\widetilde{e},p}\mathsf{L}^{1}_{\widetilde{e},p}
  \bigl({(\Poinc_{0}^{2D}\Vu\cdot\Bn_{\widetilde{e},f})}_{|\widetilde{e}}\bigr)
  \cdot \Bn_{\widetilde{e},f}\Bigr)}_{|\widetilde{e}} = \\
   \frac{d}{dl}\mathsf{L}^{1}_{\widetilde{e},p}\bigl({(\Poinc_{0}^{2D}\Vu)
   \cdot\Bn_{\widetilde{e},f}\Bigr)}_{|\widetilde{e}}
   = \Poinc_{0}^{2D}\Vu\cdot\Bn_{\widetilde{e},f}\quad\text{on }\widetilde{e}\;,
\end{multline*}
and see that the zero trace condition on $\partial f$ is satisfied. The same idea
underlies the definition of
\begin{gather}
  \label{eq:29}
  \mathsf{L}^{2}_{T,p} : 
  \left\{
    \begin{array}[c]{ccl}
      \curl\smash{\overset{\circ}{\Cw}}^{1}_{p}(T) &\mapsto & 
      \smash{\overset{\circ}{\Cw}}^{1}_{p}(T) \\
      \Vu & \mapsto & \Poinc_{0}\Vu - \grad \mathsf{E}^{0}_{\widetilde{f},p}
      \mathsf{L}^{1}_{\widetilde{f},p}\bigl(
      {((\Poinc_{0}\Vu)\times\Bn_{\widetilde{f}})}_{|\widetilde{f}}\bigr)\;,
    \end{array}
  \right.
\end{gather}
where $\widetilde{f}$ is the face opposite to vertex $0$, and the definition of
\begin{gather}
  \label{eq:109}
  \mathsf{L}^{3}_{T,p} : 
  \left\{
    \begin{array}[c]{ccl}
      \Div\smash{\overset{\circ}{\Cw}}^{2}_{p}(T) & \mapsto & 
      \smash{\overset{\circ}{\Cw}}^{2}_{p}(T) \\
      u & \mapsto & 
      \mathsf{D}_{0}u -
      \curl\mathsf{E}_{\widetilde{f},p}^{1}\mathsf{L}_{\widetilde{f},p}^{2}
      ({(\mathsf{D}_{0}u\cdot\Bn_{\widetilde{f}})}_{|\widetilde{f}})\;.
    \end{array}
  \right.
\end{gather}
The relationships between the various facet function spaces with vanishing
traces can be summarized in the following exact sequences:
\begin{gather}
  \label{eq:138}
  \begin{CD}
    \{0\} @>{\mathsf{Id}}>> \smash{\overset{\circ}{\Cp}}_{p+1}(T) 
    @>{\grad}>{\mathsf{L}_{T,p}^{1}}> \smash{\overset{\circ}{\Cw}}_{p}^{1}(T)
    @>{\curl}>{\mathsf{L}_{T,p}^{2}}> \smash{\overset{\circ}{\Cw}}_{p}^{2}(T) 
    @>{\Div}>{\mathsf{L}_{T,p}^{3}}> \overline{\Cp}_{p}(T) 
    @>{\mathsf{Id}}>> \{0\}, \\
    \{0\} @>{\mathsf{Id}}>> \smash{\overset{\circ}{\Cp}}_{p+1}(f) 
    @>{\bcurl}>{\mathsf{L}_{f,p}^{1}}> \smash{\overset{\circ}{\Cw}}_{p}^{1}(f)
    @>{\bDiv}>{\mathsf{L}^{2}_{f,p}}> \overline{\Cp}_{p}(f)
    % \smash{\overset{\circ}{\Cw}}_{p}^{2}(f) 
    @>{\mathsf{Id}}>> \{0\}, \\
    \{0\} @>{\mathsf{Id}}>> \smash{\overset{\circ}{\Cp}}_{p+1}(e) 
    @>{\frac{d}{dl}}>{\mathsf{L}^{1}_{e,p}}> \overline{\Cp}_{p}(e) 
    @>{\mathsf{Id}}>> \{0\}\;,
  \end{CD}
\end{gather}
where $\overline{\Cp}_{p}(F)$ designates degree $p$ polynomial spaces on $F$ with
vanishing mean.  These relationships and the lifting mappings are studied in
\cite[Sect.~3.4]{HIP02}.

Finally we need polynomial extension operators
\begin{align}
  \label{eq:33}
  & \mathsf{E}^{0}_{e,p}:\smash{\overset{\circ}{\Cp}}_{p+1}(e) \mapsto 
  \Cp_{p+1}(T)\;,\\
  \label{eq:36}
  & \mathsf{E}^{0}_{f,p}:\smash{\overset{\circ}{\Cp}}_{p+1}(f) \mapsto 
  \Cp_{p+1}(T)
\end{align}
that satisfy
\begin{align}
  \label{eq:34}
  {\mathsf{E}^{0}_{e,p}u}_{|e'} = 0 \quad\forall e'\in \Cf_{1}(T)\setminus\{e\}\;,\\
  \label{eq:35}
  {\mathsf{E}^{0}_{f,p}u}_{|f'} = 0\quad \forall f'\in \Cf_{2}(T)\setminus\{f\}\;.
\end{align}
Such extension operators can be constructed relying on a representation of a
polynomial on $F$, $F\in\Cf_{m}(T)$, $m=1,2$, as a homogeneous polynomial in the
barycentric coordinates of $F$, see \cite[Lemma~3.4]{HIP02}. As an alternative, one
may use the polynomial preserving extension operators proposed in \cite{MUN97a,DGS07}
and \cite{AID03}. We stress that continuity properties of the extensions
$\mathsf{E}^{l}_{F}$, $l=0,1$, $F\in\Cf_{m}(T)$, are immaterial.

\subsection{Interpolation operators}
\label{sec:interp-oper}

Now, we are in a position to define the projection based interpolation operators
locally on a generic tetrahedron $T$ with vertices $\Ba_{i}$, $i=1,2,3,4$. 

First, we devise a suitable projection (depending on the regularity parameter
$0<\epsilon<\frac{1}{2}$, which is usually suppressed to keep notations manageable)
\begin{gather}
  \label{eq:30}
  \Pi^{0}_{T,p} (=\Pi^{0}_{T,p}(\epsilon))\;:\;C^{\infty}(\overline{T})\mapsto \Cp_{p+1}(T)
\end{gather}
for degree $p$ Lagrangian $\Hone$-conforming finite elements. For $u\in
C^{0}(\overline{T})$ define ($\lambda_{i}$ is the barycentric coordinate function
belonging to vertex $\Ba_{i}$ of $T$)
\begin{align}
  \label{eq:54}
  & u^{(0)} := u - \underbrace{\sum\limits_{i=1}^{4}u(\Ba_{i})\lambda_{i}}_{:=
    w^{(0)}}\;,\\
  \label{eq:55}
  & u^{(1)} := u^{(0)} - \underbrace{\sum\limits_{e\in\Cf_{1}(T)}
    \mathsf{E}^{0}_{e,p}\mathsf{L}^{1}_{e,p}\mathsf{P}^{1}_{e,p}\frac{d}{ds}{u^{(0)}}_{|e}}_{%
    := w^{(1)}}\;,\\
  \label{eq:56}
  & u^{(2)} := u^{(1)} - 
  \underbrace{\sum\limits_{f\in\Cf_{1}(T)}\mathsf{E}^{0}_{f,p}
    \mathsf{L}^{1}_{f,p}\mathsf{P}^{1}_{f,p}\bcurl ({u^{(1)}}_{|f})}_{:= w^{(2)}}\;,\\
  \label{eq:57}
  & \Pi^{0}_{T,p}u := \mathsf{L}^{1}_{T,p}\mathsf{P}^{1}_{T,p}\grad u^{(2)} + w^{(2)} +
  w^{(1)} + w^{(0)}\;.
\end{align}
Observe that ${w^{(i)}}_{|F}=0$ for all $F\in\Cf_{m}(T)$, $0\leq m < i \leq 3$. We
point out that $w^{(0)}$ is the standard linear interpolant of $u$.
\begin{lemma}
  \label{lem:prj0}
  The linear mapping $\Pi^{0}_{T,p}$, $p\in\bbN_{0}$, is a projection onto $Cp_{p+1}(T)$
\end{lemma}
\begin{proof}
  Assume $u\in\Cp_{p+1}(T)$, which will carry over to all intermediate functions.
  Since $u^{(0)}(\Bz_{i}) = 0$, $i=1,\ldots,4$, we conclude from the projection
  property of $\mathsf{P}^{1}_{e,p}$ that
  $\mathsf{L}^{1}_{e}\mathsf{P}^{1}_{e}\frac{d}{ds}{u^{(0)}}_{|e} = {u^{(0)}}_{|e}$
  for any edge $e\in\Cf_{1}(T)$. As a consequence
  \begin{gather}
    \label{eq:37}
    u^{(1)} = u^{(0)} - \sum\limits_{e\in\Cf_{1}(T)}
    \mathsf{E}^{0}_{e,p}{u^{(0)}}_{|e}
    \quad\Rightarrow\quad u^{(1)}_{|e} = 0 \quad\forall e\in \Cf_{1}(T)\;.
  \end{gather}
  We infer $\mathsf{L}^{1}_{f,p}\mathsf{P}^{1}_{f}\bcurl ({u^{(1)}}_{|f}) =
  {u^{(1)}}_{|f}$ on each face $f\in\Cf_{2}(T)$, which implies
  \begin{gather}
    \label{eq:38}
    u^{(2)} = u^{(1)} - \sum\limits_{f\in\Cf_{1}(T)}\mathsf{E}^{0}_{f,p}({u^{(1)}}_{|f})
    \quad\Rightarrow\quad
    {u^{(2)}}_{|f} = 0 \quad\forall f\in\Cf_{2}(T)\;.
  \end{gather}
  This means that $\mathsf{L}^{1}_{T,p}\mathsf{P}^{1}_{T,p}\grad u^{(2)} = u^{(2)}$
  and finishes the proof.
\end{proof}

A similar stage by stage construction applies to edge elements and gives a projection
\begin{gather}
  \label{eq:31}
  \Pi^{1}_{T,p} (=  \Pi^{1}_{T,p}(\epsilon))\;:\;
  \boldsymbol{C}^{\infty}(\overline{T}) \mapsto \Cw^{1}(T)\;:
\end{gather}
for a directed edge $e:=[\Va_{i},\Va_{j}]$ we introduce the Whitney-1-form basis
function
\begin{gather}
  \label{eq:bf7}
  \Vb_{e} = \lambda_{i}\grad \lambda_{j} - \lambda_{j}\grad \lambda_{i}\;.
\end{gather}
These functions span $\Cw^{1}_{0}(T)$. Next, for $\Vu\in
\boldsymbol{C}^{0}(\overline{T})$ define
\begin{align}
  \label{eq:39}
  & \Vu^{(0)} := \Vu - \underbrace{\Bigl(\sum\limits_{e\in \Cf_{1}(T)}
  \int\nolimits_{e}\Vu\cdot\mathrm{d}\vec{s}\Bigr)\;\Vb_{e}}_{:= \Vw^{(0)}}\;,\\
\label{eq:40}
   & \Vu^{(1)} := \Vu^{(0)} - 
   \underbrace{
     \sum\limits_{e\in\Cf_{1}(T)} \grad \mathsf{E}^{0}_{e,p}\mathsf{L}^{1}_{e,p}
     \mathsf{P}_{e,p}^{1}({(\Vu^{(0)}\cdot\Bt_{e})}_{|e})}_{:= \Vw^{(1)}}\;,\\
   \label{eq:41}
   &  \Vu^{(2)} := \Vu^{(1)} - 
   \underbrace{
     \sum\limits_{f\in\Cf_{2}(T)}
     \mathsf{E}^{1}_{f,p}\mathsf{L}^{2}_{f,p}
     \mathsf{P}^{2}_{f,p}\bDiv({(\Vu^{(1)}\times\Bn_{f})}_{|f})
   }_{:= \Vw^{(2)}} \;,\\
   \label{eq:42}
   & \Vu^{(3)} := \Vu^{(2)} - 
   \underbrace{
     \sum\limits_{f\in\Cf_{2}(T)} \grad\mathsf{E}^{0}_{f,p}\mathsf{L}^{1}_{f,p}
     \mathsf{P}^{1}_{f,p}({(\Vu^{(2)}\times \Bn_{f})}_{|f})
   }_{:= \Vw^{(3)}} \;,\\
   \label{eq:43}
  &
  \Vu^{(4)} := \Vu^{(3)} - 
  \underbrace{
    \mathsf{L}^{2}_{T,p}\mathsf{P}_{T,p}^{2}\curl\Vu^{(3)}
  }_{:= \Vw^{(4)}} \;,\\
  \label{eq:44}
  & \Pi^{1}_{T,p}\Vu := \grad \mathsf{L}^{1}_{T,p}\mathsf{P}^{1}_{T,p}\Vu^{(4)} 
  +  \Vw^{(4)} +  \Vw^{(3)} +  \Vw^{(2)} + \Vw^{(1)} +   \Vw^{(0)}\;.
\end{align}
The contribution $\Vw^{(0)}$ is the standard interpolant $\Pi^{1}_{T,0}$ of $\Vu$
onto the local space of Whitney-1-forms (lowest order edge elements, see
\cite[Sect.~5.5.1]{MON03}). The extension operators were chosen in a way that
guarantees that $\Vw^{(2)}\cdot\Bt_{e} = 0$ and $\Vw^{(3)}\cdot\Bt_{e} = 0$ for all
$e\in\Cf_{1}(T)$.
\begin{lemma}
  \label{lem:prj1}
  The linear mapping $\Pi^{1}_{T,p}$, $p\in\bbN_{0}$, is a projection onto
  $\Cw^{1}_{p}(T)$ and satisfies the
  \emph{commuting diagram property}
  \begin{gather}
    \label{eq:32}
    \Pi^{1}_{T,p}\circ\grad = \grad\circ \Pi^{0}_{T,p}\quad\text{on }
    C^{\infty}(\overline{T})\;.
  \end{gather}
\end{lemma}
\begin{proof}
  The proof of the projection property runs parallel to that of
  Lemma~\ref{lem:prj0}. Assuming $\Vu\in\Cw^{1}_{p}(T)$ it is obvious that
  the same will hold for all $\Vu^{(i)}$ and $\Vw^{(i)}$ from
  \eqref{eq:39}-\eqref{eq:44}. In order to confirm that all projections
  can be discarded, we have to check that their arguments satisfy conditions
  of zero trace on the facet boundaries and, in some cases, belong to the kernel
  of differential operators.
  
  First, recalling the properties of the interpolation operator $\Pi^{1}_{0}$ for
  Whitney-1-forms, we find ${(\Vu^{(0)}\cdot\Bt_{e})}_{|e}\in
  \smash{\overset{\circ}{\Cw}}_{p}^{1}(e)$. This implies 
  \begin{gather}
    \label{eq:45}
    \grad \mathsf{E}^{0}_{e,p}\mathsf{L}^{1}_{e,p}
    \mathsf{P}_{e,p}^{1}({(\Vu^{(0)}\cdot\Bt_{e})}_{|e}) =
    {(\Vu^{(0)}\cdot\Bt_{e})}_{|e}\quad
    \forall e\in \Cf_{1}(T)\;,\\
    \intertext{and} 
    \label{eq:46}
    {(\Vu^{(1)}\cdot\Bt_{e})}_{|e} \equiv 0 \quad\forall e\in \Cf_{1}(T)\;.
  \end{gather}
  We see that ${(\Vu^{(1)}\times\Bn_{f})}_{|f}\in \smash{\overset{\circ}{\Cw}}_{p}^{1}(f)$ 
  for any $f\in\Cf_{2}(T)$, so that
  \begin{align}
    \label{eq:47}
    & \mathsf{P}^{2}_{f,p}\bDiv({(\Vu^{(1)}\times\Bn_{f})}_{|f}) = 
    \bDiv({(\Vu^{(1)}\times\Bn_{f})}_{|f}) \\ 
    \label{eq:48}
    \Rightarrow\quad &
    \bDiv \mathsf{L}^{2}_{f,p}
    \mathsf{P}^{2}_{f,p}\bDiv({(\Vu^{(1)}\times\Bn_{f})}_{|f}) =
    \bDiv({(\Vu^{(1)}\times\Bn_{f})}_{|f})\\
    \label{eq:49}
    \Rightarrow\quad &
    \bDiv({(\Vu^{(2)}\times\Bn_{f})}_{|f}) = 0 \quad\forall
    f\in\Cf_{2}(T)\;,\quad
    {(\Vu^{(2)}\cdot\Bt_{e})}_{|e} \equiv 0 \quad\forall e\in \Cf_{1}(T)\\
    \label{eq:50}
    \Rightarrow\quad &
    \mathsf{P}^{1}_{f,p}({(\Vu^{(2)}\times \Bn_{f})}_{|f}) = 
    {(\Vu^{(2)}\times \Bn_{f})}_{|f}\quad\forall f\in\Cf_{2}(T) \\
    \label{eq:51}
    \Rightarrow\quad &
    \grad\mathsf{E}^{0}_{f,p}\mathsf{L}^{1}_{f,p}
     \mathsf{P}^{1}_{f,p}({(\Vu^{(2)}\times \Bn_{f})}_{|f})\times\Bn_{f} = 
     {(\Vu^{(2)}\times \Bn_{f})}_{|f}\quad\forall f\in\Cf_{2}(T) \\
     \label{eq:52}
     \Rightarrow\quad &
     {(\Vu^{(3)}\times \Bn_{f})}_{|f} = 0 \quad\forall f\in\Cf_{2}(T)\\
     \label{eq:53}
     \Rightarrow\quad &
     \mathsf{P}_{T,p}^{2}\curl\Vu^{(3)} = \curl\Vu^{(3)} \\
     \label{eq:54a}
     \Rightarrow\quad &
     \curl\mathsf{L}^{2}_{T,p}\mathsf{P}_{T,p}^{2}\curl\Vu^{(3)} = \curl\Vu^{(3)}\\
     \label{eq:55a}
     \Rightarrow\quad &
     \curl\Vu^{(4)} = 0\quad\Rightarrow\quad
     \mathsf{P}^{1}_{T}\Vu^{(4)} = \Vu^{(4)} \\
     \label{eq:56a}
     \Rightarrow\quad &
     \grad \mathsf{L}^{1}_{T}\mathsf{P}^{1}_{T}\Vu^{(4)} = \Vu^{(4)}\;,
  \end{align}
  which confirms the projector property.

  Now assume $\Vu=\grad u$ for some $u\in C^{\infty}(\overline{T})$.  The commuting
  diagram property will follow, if we manage to show $\grad u^{(0)} = \Vu^{(0)}$,
  $\grad u^{(1)} = \Vu^{(1)}$, $\grad u^{(2)} = \Vu^{(3)}$, etc., for the intermediate
  functions in \eqref{eq:54}-\eqref{eq:57} and \eqref{eq:39}-\eqref{eq:44},
  respectively.

  By the commuting diagram property for the standard local interpolation operators
  onto the spaces of Whitney-0-forms (linear polynomials) and Whitney-1-forms, we
  conclude 
  \begin{align}
    \label{eq:58}
    & \grad u^{(0)} = \Vu^{(0)}\quad\Rightarrow\quad
    \frac{d}{ds}{u^{(0)}}_{|e} = 
    {(\Vu^{(0)}\cdot\Bt_{e})}_{|e}\quad \forall e\in\Cf_{1}(T) \\
    \label{eq:59}
    \Rightarrow\quad &
    \Vu^{(1)} = \grad u^{(1)}\quad\Rightarrow\quad
    \bDiv({(\Vu^{(1)}\times\Bn_{f})}_{|f}) = 0 \quad\forall f\in\Cf_{2}(T) 
    \\
    \label{eq:60}
    \Rightarrow\quad & \Vu^{(2)} = \Vu^{(1)}\\
     \label{eq:60a}
    \Rightarrow\quad &
    {(\Vu^{(2)}\times \Bn_{f})}_{|f} = \bcurl {u^{(1)}}_{f}\quad\forall
    f\in\Cf_{2}(T)\quad
    \Rightarrow\quad \Vu^{(3)} = \grad u^{(2)}\\
    \label{eq:61}
    \Rightarrow\quad & \Vu^{(4)} =
    \Vu^{(3)}\;.
  \end{align}
  Of course, analogous relationships for the functions $w^{(i)}$ and $\Vw^{(i)}$
  hold, which yields $\Pi^{1}_{T,p}\Vu = \grad \Pi^{0}_{T,p}u$. 
\end{proof}

Following \cite[Sect.~3.5]{HIP02}, a projection based interpolation onto
$\Cw^{2}_{p}(T)$, the operator $\Pi^{2}_{T,p}
(=\Pi^{2}_{T,p}(\epsilon)):\boldsymbol{C}^{\infty}
(\overline{T})\mapsto\Cw^{2}_{p}(T)$,
involves the stages
\begin{align}
  \label{eq:20}
  \Vu^{(0)} & := \Vu - 
  \underbrace{\Bigl(\sum\limits_{f\in\Cf_{2}(T)} \int\nolimits_{f}
    \Vu\cdot\Bn_{f}\,\mathrm{d}S\Bigr)\,\Vb_{f}}_{:= \Vw^{(0)}}\;,\\
  \label{eq:110}
  \Vu^{(1)} & := \Vu^{(0)} - 
  \underbrace{
    \sum\limits_{f\in\Cf_{2}(T)} \curl\mathsf{E}_{f,p}^{1}\mathsf{L}_{f,p}^{2}
    \mathsf{P}_{f,p}^{2}\bigl(
    {(\Vu^{(0)}\cdot\Bn_{f})}_{|f}
    \bigr)
  }_{{:=\Vw^{(1)}}} \\
  \label{eq:111}
  \Vu^{(2)} & := \Vu^{(1)} - \underbrace{
    \mathsf{L}^{3}_{T,p}\mathsf{P}^{3}_{T,p}\Div\Vu^{(1)}}_{:=\Vw^{(2)}}\\
  \label{eq:112}
  \Pi^{2}_{T,p}\Vu & := \curl\mathsf{L}^{2}_{T,p}\mathsf{P}_{T,p}\Vu^{(2)}
  + \Vw^{(0)} + \Vw^{(1)} + \Vw^{(2)}\;.
\end{align}
Here, $\Vb_{f}$ refers to the local basis functions for Whitney-2-forms
\cite[Sect.~3.2]{HIP02}:
\begin{gather}
  \label{eq:114}
  \Vb_{f} = \lambda_{i}\grad\lambda_{j}\times\grad \lambda_{k} + 
  \lambda_{j}\grad\lambda_{k}\times\lambda_{i} + 
  \lambda_{k}\grad\lambda_{i}\times\lambda_{j}\;.
\end{gather}

Analogous to Lemma~\ref{lem:prj1} one proves the following result.

\begin{lemma}
  \label{lem:43}
  The linear operator $\Pi^{2}_{T,p}$, $p\in\bbN_{0}$, is a projection onto
  $\Cw^{2}_{p}(T)$ and satisfies the \emph{commuting diagram property}
  \begin{gather}
    \label{eq:113}
    \Pi^{2}_{T,p}\circ \curl = \curl \circ\Pi^{1}_{T,p}\quad
    \text{on }\boldsymbol{C}^{\infty}(\overline{T})\;.
  \end{gather}
\end{lemma}

The next lemma makes it possible to patch together the local projection
based interpolation operator to obtain global interpolation operators
\begin{gather}
  \label{eq:118}
  \Pi_{p}^{l}:\boldsymbol{C}^{\infty}(\overline{\Omega}) \mapsto
  \Cw^{l}_{p}(\mesh)\;,\quad
  l=1,2\;.
\end{gather}

\begin{lemma}
  \label{lem:comp}
  For any $F\in\Cf_{m}(T)$, $m=0,1,2$, and $u\in C^{\infty}(\overline{T})$ 
  the restriction ${\Pi^{0}_{T,p}u}_{|F}$ depends only on $u_{|F}$. 

  For any $F\in\Cf_{m}(T)$, $m=1,2$, and $\Vu\in \boldsymbol{C}^{\infty}(\overline{T})$ 
  the tangential trace of $\Pi^{1}_{T,p}\Vu$ onto $F$ depends only on the tangential
  trace of $\Vu$ on $F$. 

  For any face $f\in\Cf_{2}(T)$ and $\Vu\in \boldsymbol{C}^{\infty}(\overline{T})$ 
  the normal trace of $\Pi^{2}_{T,p}\Vu$ onto $f$ depends only on the normal
  component of $\Vu$ on $f$.
\end{lemma}

\begin{proof}
  The assertion is immediate from the construction, in particular, the properties
  of the extension operators used therein. 
\end{proof}

It goes without saying that density arguments permit us to extend $\Pi^{l}_{p}$,
$l=0,1,2$, to Sobolev spaces, as long as they are continuous in the respective
norms. (Repeated) application of trace theorems \cite[Sect.~1.5]{GRI92} reveals that
it is possible to obtain continuous projectors
\begin{align}
  \label{eq:139}
  \Pi^{0}_{p}&:\Hm{1+s}\mapsto \{v\in\Hone:\; v_{|T}\in\Cp_{p+1}(T)\;\forall T\in\mesh\}\;,\\
  \label{eq:141}
  \Pi^{1}_{p}&:\Hmv{\frac{1}{2}+s} \mapsto \Cw^{1}_{p}(\mesh)\;,\\
  \label{eq:142}
  \Pi^{2}_{p}&:\Hmv{s}\mapsto \Cw^{2}_{p}(\mesh)\;,
\end{align}
for any $s>\frac{1}{2}$. In addition, by virtue of Lemma~\ref{lem:comp}, zero
pointwise/tangential/normal trace on $\partial\Omega$ of the argument function will
be preserved by $\Pi^{l}_{p}$, $l=0,1,2$, for instance,
\begin{gather}
  \label{eq:146}
  \Pi^{1}_{p}(\Hmv{\frac{1}{2}+s}\cap\zbHcurl) = \Cw^{1}_{p}(\mesh)\cap\zbHcurl\;.
\end{gather}

\section{Interpolation error estimates}
\label{sec:interp-estim}

Closely following the ingenious approach in \cite[Section~6]{DEB04} we first examine
the interpolation error for $\Pi^{0}_{T,p}$. Please notice that $\Pi^{0}_{T,p}$
still depends on the fixed regularity parameter $0<\epsilon<\frac{1}{2}$. The argument 
function of $\Pi^{0}_{T,p}$ is assumed to lie in $\Hm[T]{1+s}$ for some
$s>\frac{1}{2}$, \textit{cf.} \eqref{eq:139}. The continuous embedding $\Hm[T]{1+s}\hookrightarrow
C^{0}(\overline{T})$ plus trace theorems for Sobolev spaces render all operators
well defined in this case.

We start with an observation related to the local best approximation properties of
the projection based interpolant.

\begin{lemma}
  \label{lem:51}
  For any $u\in\Hm[T]{1+s}$ holds
  \begin{align}
    \label{eq:62}
    \left(\grad(u-\Pi^{0}_{T,p}u),\grad v\right)_{\Ltwo[T]} & =  0 \quad \forall v\in
    \smash{\overset{\circ}{\Cp}}_{p+1}(T)\;,\\
    \label{eq:67}
    \left(\bcurl 
      {(u-\Pi^{0}_{T,p}u)}_{|f},\bcurl v\right)_{H^{-\frac{1}{2}+\epsilon}(f)} & =
    0 
    \quad\forall v\in \smash{\overset{\circ}{\Cp}}_{p+1}(f),\;f\in\Cf_{2}(T)\;,\\
    \label{eq:68}
    \left(\frac{d}{dl}{(u-\Pi^{0}_{T,p}u)}_{|e},\frac{d}{dl}
      v\right)_{H^{-1+\epsilon}(e)} & =  0 
    \quad\forall v\in \smash{\overset{\circ}{\Cp}}_{p+1}(e),\;e\in\Cf_{1}(T)\;.
  \end{align}
\end{lemma}

\begin{proof} We use the notations of \eqref{eq:54}-\eqref{eq:57}. Setting 
  $w:= w^{(0)}+w^{(1)}+w^{(2)}$, we find
  \begin{gather}
    \label{eq:64}
    \Pi^{0}_{T,p}u = \mathsf{L}_{T,p}^{1}\mathsf{P}_{T,p}^{1}\grad(u-w)+w\;,
  \end{gather}
  which implies, because $\mathsf{L}_{T,p}^{1}$ is a right inverse of $\frac{d}{dl}$,
  \begin{gather}
    \label{eq:65}
    \grad \Pi^{0}_{T,p}u = \mathsf{P}_{T,p}^{1}\grad u +
    (\mathsf{Id}-\mathsf{P}_{T,p}^{1})\grad w\;.
  \end{gather}
  This means that $u-\grad \Pi^{0}_{T,p}u$ belongs to the range of
  $\mathsf{Id}-\mathsf{P}_{T,p}^{1}$ and \eqref{eq:62} follows from \eqref{eq:23} and
  the properties of orthogonal projections. Similar manipulations establish
  \eqref{eq:67}:
  \begin{align}
    \label{eq:66}
    {\bcurl\Pi^{0}_{T,p}u}_{|f} & = 
    {\bcurl w}_{|f} \\
    \label{eq:69}
    & = \underbrace{\bcurl\mathsf{L}_{f,p}^{1}}_{=\mathsf{Id}}
    \mathsf{P}^{1}_{f,p}\bcurl u^{(1)} + 
    {\bcurl(w^{(0)}+w^{(1)})}_{|f} \\
    \label{eq:70}
    & = \mathsf{P}^{1}_{f,p}\bcurl {u}_{|f}+
    (\mathsf{Id}-\mathsf{P}^{1}_{f,p})\bcurl(w^{(0)}+w^{(1)})\quad\forall f\in\Cf_{2}(T)\;.
  \end{align}
  The same arguments as above verify \eqref{eq:68}.
\end{proof}

From this we can conclude the result of \cite[Section~6, Corollary~1]{DEB04}. To
state it we now assume a dependence 
\begin{gather}
  \label{eq:133}
  0<\epsilon = \epsilon(p) := \frac{1}{10\log(p+1)} < \frac{1}{4}\;,\quad
  p \in \bbN\;,
\end{gather}
of the parameter $\epsilon$ in the definition of the local projection based
interpolation operators. Below, all parameters $\epsilon$ are linked to $p$ via
\eqref{eq:133}.  Please note that we retain the notation
$\left(\Pi_{T,p}^{l}\right)_{p\in\bbN}$, $l=0,1,2$, for these new families of
operators. 

\begin{theorem}[Spectral interpolation error estimate for $\Pi^{0}_{T,p}$]
  \label{thm:deb04}
  There is a constant $C_{T}>0$ depending only on $T$ and $\frac{1}{2}<s\leq 1$, and,
  in particular, independent of $p$, such that
  \begin{gather}
    \label{eq:63}
    \SNHone[T]{(\Op{Id}-\Pi_{T,p}^{0})v} \leq C_{T} 
    \frac{\log^{3/2}p}{p^{s}}\SNHm[T]{v}{1+s}\quad
    \forall v\in \Hm[T]{1+s}\;,\quad p\geq 1\;.
  \end{gather}
\end{theorem}

Stable polynomial extensions are instrumental for the proof, which will be postponed
until Page~\pageref{pdeb04}. First, we recall the results of \cite[Thm.~1]{MUN97a}
and \cite[Thm.~1]{AID03}:

\begin{theorem}[Stable polynomial extension for tetrahedra]
  \label{thm:MUN}
  For a tetrahedron $T$ there is linear operator $\mathsf{S}_{T}:\Hh[\partial
  T]\mapsto \Hone[T]$ such that 
  \begin{align}
    \label{eq:74}
    {\mathsf{S}_{T}u}_{|\partial T} & = u  \quad\forall u\in \Hh[\partial T]\;,\\
    \label{eq:73}
    \SNHone[T]{\mathsf{S}_{T}u} & \leq C \SNHh[\partial T]{u}  \quad \forall u\in
    \Hh[\partial T]\;,\\
    \label{eq:76}
    \mathsf{S}_{T}w & \in \Cp_{p+1}(T)  \quad\forall w\in {\Cp_{p+1}(T)}_{|\partial T}\;,
  \end{align}
  where $C>0$ only depends on the shape regularity measure\footnote{The shape
    regularity measure of a tetrahedron is the ratio of the radii of its circumscribed
    sphere and the largest inscribed sphere.} of $T$.
\end{theorem}

\begin{theorem}[Stable polynomial extension for triangles]
  \label{thm:AID}
  Given a triangle $F$, there is a \emph{continuous} linear mapping $\mathsf{S}_{F}:\Ltwo[\partial
  F]\mapsto \Hh[T]$ such that
  \begin{align}
    \label{eq:75}
    \SNHone[F]{\mathsf{S}_{F}u} & \leq C \SNHh[\partial F]{u}\quad\forall u\in
    \Hh[\partial F]\;,\\
    \label{eq:77}
    \mathsf{S}_{F}w & \in \Cp_{p+1}(F)  \quad\forall w\in {\Cp_{p+1}(F)}_{|\partial F}\;,
  \end{align}
  where $C>0$  depends only on the shape regularity measure of $T$.
\end{theorem}

By interpolation in Sobolev scale from the last theorem we can conclude 
\begin{gather}
  \label{eq:78}
  \exists C>0:\quad \SNHm[F]{\mathsf{S}_{F}u}{s} \leq 
  C \SNHm[\partial F]{u}{s-\frac{1}{2}}\quad\forall u\in \Hm[\partial
  F]{s-\frac{1}{2}},\;
  \frac{1}{2}\leq s \leq 1\;.
\end{gather}

We also need to deal with the awkward property of the $\Hh[\partial T]$-norm
that it cannot be localized to faces. To that end we resort to a result from 
\cite[Proof of Lemma 3.31]{MCL00}, see also \cite[Lemma~13]{DEB04}.

\newcommand{\Hhpt}{\Hh[\partial T]}
\begin{lemma}[Splitting of $\Hhpt$-norm]
  \label{lem:hhptsplit}
  The exists $C>0$ depending only on the shape regularity of the tetrahedron $T$ such
  that
  \begin{gather}
    \label{eq:79}
    \SNHm[\partial T]{u}{\frac{1}{2}+\epsilon} \leq \frac{C}{\epsilon}
    \sum\limits_{f\in\Cf_{2}(T)} \SNHm[f]{u}{\frac{1}{2}+s}\quad \forall u\in\Hm[\partial
    T]{\frac{1}{2}+\epsilon},\;
    0<\epsilon\leq \frac{1}{2}\;.
  \end{gather}
\end{lemma}

Another natural ingredient for the proof are polynomial best approximation
estimates, see \cite{SAB98} or \cite[Sect.~3]{MUN97a}.

\begin{lemma}
  \label{lem:Tpolbest}
  Let $F$ be either a tetrahedron or a triangle. Then, there is a constant $C>0$ depending
  only on $F$ such that for all $p\geq 1$
  \begin{gather}
    \label{eq:80}
    \inf\limits_{v_{p}\in \Cp_{p+1}(F)} \SNHm[F]{u-v_{p}}{r} \leq Cp^{r-1-s}
    \SNHm[F]{u}{1+s}\quad\forall u\in \Hm[F]{1+s},\; 0\leq r \leq 1\;.
  \end{gather}
\end{lemma}

% \begin{lemma}
%   \label{lem:Ipolbest}
%   For any interval $I\subset \bbR$
%   \begin{gather}
%     \label{eq:81}
%     \inf\limits_{v_{p}\in\Cp_{p+1}(I)\cap\zbHone[I]}\SNHm[I]{u-v_{p}}{r}
%     \leq C p^{r-1-s} \SNHm[I]{u}{1+s}\quad\forall u\in \Hm[I]{1+s}\cap\zbHone[I]
%     \,\; 0\leq r \leq s \leq 1\;.
%   \end{gather}
% \end{lemma}

Define a semi-norm projection $\mathsf{Q}_{T,p}:\Hone[T]\mapsto \Cp_{p+1}(T)$ on the
tetrahedron $T$ by
\begin{gather}
  \label{eq:82}
  \begin{aligned}
    & \int\nolimits_{T}\grad(u-\mathsf{Q}_{T,p}u)\cdot\grad v_{p}\,\mathrm{d}\Bx = 0
    \quad\forall
    v_{p}\in\Cp_{p+1}(T)\;,\\
    & \int\nolimits_{T}u-\mathsf{Q}_{T,p}u\,\mathrm{d}\Bx = 0 \;,
  \end{aligned}
\end{gather}
and semi-norm projections
$\mathsf{Q}_{f,p}:\Hm[f]{s+\frac{1}{2}}\mapsto\Cp_{p+1}(f)$, $f\in\Cf_{2}(T)$, by
\begin{gather}
  \label{eq:87}
  \begin{aligned}
    & \left(\bcurl(u-\mathsf{Q}_{f,p}u),\bcurl v_{p}\right)_{\Hm[f]{\epsilon-\frac{1}{2}}} = 0
    \quad\forall
    v_{p}\in\Cp_{p+1}(T)\;,\\
    & \int\nolimits_{f}u-\mathsf{Q}_{f,p}u\,\mathrm{d}\Bx = 0 \;.
  \end{aligned}
\end{gather}
These definitions involve best approximation properties of $\mathsf{Q}_{T,p}u$ and
$\mathsf{Q}_{f,p}u$. Thus, we learn from Lemma~\ref{lem:Tpolbest} that with
constants independent of $0<\epsilon<\frac{1}{2}<s\leq 1$
\begin{align}
  \label{eq:144}
  \SNHone[T]{u-\mathsf{Q}_{T,p}u} & \leq C (p+1)^{-s}\SNHm[T]{u}{1+s}\quad
  \forall u\in\Hm[T]{s}\;,\\
  \label{eq:145}
  \SNHm[f]{u-\mathsf{Q}_{f,p}u}{\frac{1}{2}+\epsilon} & \leq C (p+1)^{\epsilon-s}
  \SNHm[T]{u}{\frac{1}{2}+s}\quad\forall u\in\Hm[f]{\frac{1}{2}+s}\;.
\end{align}
The latter estimate follows from the fact that
$\SNHm[f]{\cdot}{\frac{1}{2}+\epsilon}$ and
$\NHm[f]{\bcurl\cdot}{-\frac{1}{2}+\epsilon}$ are equivalent  semi-norms, uniformly in
$\epsilon$.

We also need error estimates for the $\Ltwo[e]$-orthogonal projections,
\begin{gather}
  \label{eq:89}
  \mathsf{Q}^{\ast}_{e,p}:\Ltwo[e] \mapsto \smash{\overset{\circ}{\Cp}}_{p+1}(e)\;,\quad
  e\in\Cf_{1}(T)\;.
\end{gather}
\newcommand{\debl}{\cite[Lemma~18]{DEB04}}
\begin{lemma}[see \debl]
  \label{lem:Qepest}
  With a constant $C>0$ independent of $p$, $0\leq \epsilon\leq \frac{1}{2}$, and 
  $2\epsilon \leq r \leq 1+\epsilon$
  \begin{gather*}
    \SNHm[e]{e-\mathsf{Q}^{\ast}_{e,p}u}{\epsilon} \leq C(p+1)^{2\epsilon-r}\SNHm[e]{u}{r}\quad
    \forall u\in \Hm[e]{r}\cap\zbHone[e]\;.
  \end{gather*}
\end{lemma}
\begin{proof}
  Write $\mathsf{I}_{e,p}:\zbHone[e]\mapsto \smash{\overset{\circ}{\Cp}}_{p+1}$ for the
  interpolation operator
  \begin{gather*}
    (\mathsf{I}_{e,p}u)(\xi) = u(0) + \int\nolimits_{0}^{\xi}
    \bigl(\mathsf{Q}_{e,p}\frac{du}{d\xi}\bigr)(\tau)\,\mathrm{d}\tau\;,\quad
    0\leq \xi \leq |e|\;,
  \end{gather*}
  where $\xi$ is the arclength parameter for the edge $e$ and
  $\mathsf{Q}_{e,p}:\Ltwo\mapsto\Cp_{p}(e)$ is the $\Ltwo[e]$-orthogonal projection.
  From \cite[Sect.~3.3.1, Thm.~3.17]{SAB98} we learn that
  \begin{align}
    \label{eq:90}
    \SNHone[e]{u-\mathsf{I}_{e,p}u} & \leq C (p+1)^{-1}\SNHm[e]{u}{2}\quad
    \forall u\in\Hm[e]{2}\;,\\
    \label{eq:91}
    \NLtwo[e]{u-\mathsf{I}_{e,p}u} & \leq C (p+1)^{-m}\SNHm[e]{u}{m}\quad
    \forall u\in \Hm[e]{m}\;,\quad m=1,2\;.
  \end{align}
  Here and the in the remainder of the proof, all constants may depend only on the
  length of $e$. As $\mathsf{I}_{e,p}u\in\smash{\overset{\circ}{\Cp}}_{p+1}(e)$ for
  $u\in\zbHone[e]$, this permits us to conclude
  \begin{gather}
    \label{eq:92}
    \NLtwo[e]{u-\mathsf{Q}^{\ast}_{e,p}u} \leq 
    \NLtwo[e]{u-\mathsf{I}_{e,p}u} \leq C (p+1)^{-1}\NHone[e]{u}\;,
  \end{gather}
  which yields, by interpolation between $\Hone[e]$ and $\Ltwo[e]$,
  \begin{gather}
    \label{eq:93}
    \NLtwo[e]{u-\mathsf{Q}^{\ast}_{e,p}u} \leq C (p+1)^{-q}\NHm[e]{u}{q}\;,\quad
      0\leq q \leq 1\;,
  \end{gather}
  where $C>0$ is independent of $q$. On the other hand, using the inverse
  inequality \cite[Lemma~1]{BEM97b}
  \begin{gather}
    \label{eq:132}
    \NHone[e]{u} \leq C (p+1)^{2} \NLtwo[e]{u} \quad\forall u\in\Cp_{p+1}(e)\,
  \end{gather}
  and \eqref{eq:90}, \eqref{eq:91} we find the estimate
  \begin{gather}
    \label{eq:94}
    \begin{aligned}
      \SNHone[e]{u-\mathsf{Q}^{\ast}_{e,p}u} & \leq 
      \SNHone[e]{u-\mathsf{I}_{e,p}u} + \SNHone[e]{\mathsf{Q}^{\ast}_{e,p}u-\mathsf{I}_{e,p}u}
      \\
      & \leq \SNHone[e]{u-\mathsf{I}_{e,p}u} +
      (p+1)^{2}\NLtwo[e]{\mathsf{Q}^{\ast}_{e,p}u-\mathsf{I}_{e,p}u} \\
      & \leq \SNHone[e]{u-\mathsf{I}_{e,p}u} +
      C(p+1)^{2}\NLtwo[e]{u-\mathsf{I}_{e,p}u} \\
      & \leq C \NHm[e]{u}{2}\;.
    \end{aligned}
  \end{gather}
  Interpolation between \eqref{eq:93} with $q=\frac{r-2\epsilon}{1-\epsilon}$ and \eqref{eq:94}
  finishes the  proof.
\end{proof}

\newcommand{\deb}{\cite[Sect.~6]{DEB04}}
\begin{proof}[of Thm.~\ref{thm:deb04}, borrowed from \deb]
  \label{pdeb04}
  Orthogonality \eqref{eq:62} of Lemma~\ref{lem:51} combined with the
  definition of $\mathsf{Q}_{T,p}$ involves
  \begin{gather}
    \label{eq:83}
    \int\limits_{T}\grad((\Pi_{T,p}^{0}-\mathsf{Q}_{T,p})u)\cdot\grad v_{p}\,\mathrm{d}\Bx = 0
    \quad\forall
    v_{p}\in\smash{\overset{\circ}{\Cp}}_{p+1}(T)\;.
  \end{gather}
  Hence, $(\Pi_{T,p}^{0}-\mathsf{Q}_{T,p})u$ turns out to be the
  $\SNHone[T]{\cdot}$-minimal degree $p+1$ polynomial extension of
  ${(\Pi_{T,p}^{0}-\mathsf{Q}_{T,p})u}_{|\partial T}$, which,thanks to
  Thm.~\ref{thm:MUN}, implies
  \begin{gather}
    \label{eq:84}
    \begin{aligned}
      \SNHone[T]{(\Pi_{T,p}^{0}-\mathsf{Q}_{T,p})u} & \leq 
      \SNHone[T]{\mathsf{S}_{T}({(\Pi_{T,p}^{0}u-\mathsf{Q}_{T,p}u)}_{|\partial T})}
      \\
      & \leq C \SNHh[\partial T]{{(\Pi_{T,p}^{0}u-\mathsf{Q}_{T,p}u)}_{|\partial T}}\;.
    \end{aligned}
  \end{gather}
  Thus, by the continuity of the trace operator $\Hone[T]\mapsto\Hh[\partial T]$,
  \begin{gather}
    \label{eq:85}
    \begin{aligned}
      \SNHone[T]{u-\Pi_{T,p}^{0}u} & \leq 
      \begin{aligned}[t]
        & \SNHone[T]{u-\mathsf{Q}_{T,p}u} \\ & + C
        \left(\SNHh[\partial T]{{(u-\Pi_{T,p}^{0}u)}_{|\partial T}}+
          \SNHh[\partial T]{{(u-\mathsf{Q}_{T,p}u)}_{|\partial T}}\right)
      \end{aligned}
      \\
      & \leq C \left(\SNHone[T]{u-\mathsf{Q}_{T,p}u} + 
        \SNHh[\partial T]{{(u-\Pi_{T,p}^{0}u)}_{|\partial T}}\right)\;.
    \end{aligned}
  \end{gather}
  To estimate $\SNHh[\partial T]{{(u-\Pi_{T,p}^{0}u)}_{|\partial T}}$ we appeal to 
  Lemma~\ref{lem:hhptsplit} and get 
  \begin{gather}
    \label{eq:86}
    \begin{aligned}
      \SNHh[\partial T]{{(u-\Pi_{T,p}^{0}u)}_{|\partial T}} & \leq
      \SNHm[\partial T]{{(u-\Pi_{T,p}^{0}u)}_{|\partial T}}{\frac{1}{2}+\epsilon} \\ & \leq
      \frac{C}{\epsilon}\sum\limits_{f\in\Cf_{2}(T)} 
      \SNHm[f]{{(u-\Pi_{T,p}^{0}u)}_{|f}}{\frac{1}{2}+\epsilon}\;.
    \end{aligned}
  \end{gather}
  Next, we use \eqref{eq:67} from Lemma~\ref{lem:51} together with \eqref{eq:87},
  which confirms that ${(\Pi_{T,p}^{0}u)}_{|f}-\mathsf{Q}_{f,p}u$ is the minimum
  $\SNHm[f]{\cdot}{\frac{1}{2}+\epsilon}$-seminorm polynomial extension of
  ${(\Pi_{T,p}^{0}u)}_{|\partial f}-\mathsf{Q}_{f,p}{(u)}_{|\partial f}$. Hence,
  based on arguments parallel to the derivation of \eqref{eq:85}, this time using
  Thm.~\ref{thm:AID}, we can bound
  \begin{gather}
    \label{eq:88}
      \SNHm[f]{{(u-\Pi_{T,p}^{0}u)}_{|f}}{\frac{1}{2}+\epsilon} \leq 
      \SNHm[f]{u_{|f}-\mathsf{Q}_{f,p}u}{\frac{1}{2}+\epsilon} + C
        \SNHm[\partial f]{{(\Pi_{T,p}^{0}u-\mathsf{Q}_{f,p}u)}_{|\partial
            f}}{\epsilon}\;,
  \end{gather}
  where the ($\epsilon$-independent !) continuity constant of the trace mapping
  $\mathsf{S}_{f}$ enters the constant $C>0$. Also recall the continuity of the trace
  mapping $\Hm[f]{\frac{1}{2}+\epsilon}\mapsto\Hm[\partial f]{\epsilon}$ \cite[Proof
  of Lemma 3.35]{MCL00}: with $C>0$ independent of $\epsilon$,
  \begin{gather}
    \label{eq:130}
    \NHm[\partial f]{u_{|\partial f}}{\epsilon} \leq \frac{C}{\sqrt{\epsilon}}
    \NHm[f]{u}{\frac{1}{2}+\epsilon}\quad\forall u\in\Hm[f]{\frac{1}{2}+\epsilon}\;.
  \end{gather}
  Use this to continue the estimate \eqref{eq:88}
  \begin{gather}
    \label{eq:131}
    \SNHm[f]{{(u-\Pi_{T,p}^{0}u)}_{|f}}{\frac{1}{2}+\epsilon}
    \leq C \left(\frac{1}{\sqrt{\epsilon}}
       \SNHm[f]{u_{|f}-\mathsf{Q}_{f,p}u}{\frac{1}{2}+\epsilon} + 
       \SNHm[\partial f]{{(u-\Pi_{T,p}^{0}u)}_{|\partial
           f}}{\epsilon}
     \right)\;.
   \end{gather}
   As $\epsilon<\frac{1}{2}$, we can localize the norm $\SNHm[\partial
   f]{{(u-\Pi_{T,p}^{0}u)}_{|\partial f}}{\epsilon}$ to the edges of $f$:
  \begin{gather}
    \label{eq:95}
    \SNHm[\partial
    f]{{(u-\Pi_{T,p}^{0}u)}_{|\partial f}}{\epsilon} \leq
    \frac{C}{\frac{1}{2}-\epsilon} \sum\limits_{e\in\Cf_{1}(T), e\subset\partial f}
    \SNHm[e]{{(u-\Pi_{T,p}^{0}u)}_{|e}}{\epsilon}\;.
  \end{gather}
  Recall the $\epsilon$-uniform equivalence of the norms $\SNHm[e]{\cdot}{\epsilon}$
  and $\NHm[e]{\frac{d}{dl}\cdot}{-1+\epsilon}$. Hence, owing to \eqref{eq:68}, we have
  from Lemma~\ref{lem:Qepest} with $r=s$:
  \begin{gather}
    \label{eq:96}
    \begin{aligned}
      \SNHm[e]{{(u-\Pi_{T,p}^{0}u)}_{|e}}{\epsilon} & \leq C
      \inf\limits_{v_{p}\in \overset{\circ}{\Cp}_{p+1}}
      \SNHm[e]{{(u-\Pi_{T,0}^{0}u)}_{|e} - v_{p}}{\epsilon} \\
      & \leq C \SNHm[e]{{(u-\Pi_{T,0}^{0}u)}_{|e} -
        \mathsf{Q}^{\ast}_{e,p}({(u-\Pi_{T,0}^{0}u)}_{|e})}{\epsilon} \\
      & \leq C(p+1)^{2\epsilon-s}\SNHm[e]{{(u-\Pi_{T,0}^{0}u)}_{|e}}{s}\;.
    \end{aligned}
  \end{gather}
  Moreover, $\Hm[T]{1+s}$ is continuously embedded into $C^{0}(\overline{T})$. 
  Consequently, applying trace theorems twice and appealing to the equivalence
  of all norms on the finite dimensional space $\Cp_{1}(T)$,
  \begin{gather}
    \label{eq:119}
    \SNHm[e]{{(u-\Pi_{T,0}^{0}u)}_{|e}}{s} \leq \SNHm[e]{{u}_{|e}}{s} + 
    \SNHm[e]{{(\Pi_{T,0}^{0}u)}_{|e}}{s}
    \leq C\SNHm[T]{u}{1+s}\;,
  \end{gather}
  where $C>0$ depends on $s$ and $T$, but not on $p$. Combining the estimates
  \eqref{eq:85}, \eqref{eq:86}, \eqref{eq:131}, and \eqref{eq:95}, \eqref{eq:96} with
  \eqref{eq:119}, we find
  \begin{gather}
    \label{eq:120}
    \SNHone[T]{u-\Pi^{0}_{T,p}u} \leq C 
    \Bigl(
    \begin{aligned}[t] & 
      \SNHone[T]{u-\mathsf{Q}_{T,p}u} + \frac{1}{\epsilon^{3/2}}\sum\limits_{f\in\Cf_{2}(T)}
      \SNHm[f]{u_{|f}-\mathsf{Q}_{f,p}({u}_{|f})}{\frac{1}{2}+\epsilon} + \\ & 
      \frac{(p+1)^{2\epsilon-s}}{\epsilon(\frac{1}{2}-\epsilon)}
      \sum\limits_{e\in\Cf_{1}(T)}\SNHm[T]{u}{1+s}
    \Bigr)\;,
  \end{aligned}
  \end{gather}
  with $C>0$ independent of $p$. Finally, we plug in the projection error
  estimates \eqref{eq:144}, \eqref{eq:145}, and arrive at ($C>0$ independent
  of $u$, $\epsilon$, $p$, $s$)
  \begin{gather}
    \label{eq:121}
    \SNHone[T]{u-\Pi^{0}_{T,p}(\epsilon)u} \leq C 
    \Bigl(
    \begin{aligned}[t] & 
      (p+1)^{-s}\SNHm[T]{u}{1+s} + \frac{(p+1)^{-s+\epsilon}}{\epsilon^{3/2}}
      \sum\limits_{f\in\Cf_{2}(T)} 
      \SNHm[f]{u}{\frac{1}{2}+s} + 
      \\ & 
      \frac{(p+1)^{-s+2\epsilon}}{\epsilon(\frac{1}{2}-\epsilon)}
      \sum\limits_{e\in\Cf_{1}(T)}\SNHm[e]{u}{s}
    \Bigr)\;.
  \end{aligned}
  \end{gather}
  The choice \eqref{eq:133} of $\epsilon$ together with an application of trace
  theorems then finishes the proof.
\end{proof}

The next lemma plays the role of \cite[Lemma~9]{BCD06} and makes it possible to
adapt the approach of \cite[Sect.~4.4]{BCD06} to 3D edge elements. 

\begin{lemma} 
  \label{lem:1}
  If $\frac{1}{2}<s\leq 1$ and $\Vu\in\Hmv{s}$ satisfies $\curl\Vu_{|T}\in \Cpv_{p}(T)$ for all
  $T\in\mathcal{M}$, then
  \begin{gather}
    \label{eqo:6}
    \NLtwo{
      %h_{\mesh}^{-s}
      (\Op{Id}-\Pi_{p}^{1})\Vu} \leq C\; \frac{\log^{3/2}p}{p^{s}} \bigl(\NHm{\Vu}{s} + 
    \NLtwo[T]{\curl\Vu}\bigr)\;,
  \end{gather}
  with a constant $C>0$ depending only on $\mathcal{M}$ and $s$.
\end{lemma}

\begin{proof}
Pick any $\Vu$ complying with the assumptions of the lemma. The
locality of the projector allows purely local considerations. Single out one
tetrahedron $T\in\mathcal{M}$, still write $\Vu={\Vu}_{|T}$, and split on $T$
\begin{gather}
  \label{eqo:7}
  \Vu ={(\Vu-\Op{R}_{T}\curl\Vu)} +  {\Op{R}_{T}\curl\Vu}\;,
\end{gather}
Note that the properties of the smoothed Poincar\'e lifting 
$\Poinc_{T}$ stated in Thm.~\ref{thm:CMI} imply
\begin{enumerate}
\item $\displaystyle \curl(\Vu-\Op{R}_{T}\curl\Vu)=0$ on $T$, as a consequence of
  \eqref{eq:cmi}, and
\item ${\Op{R}_{T}\curl\Vu}\in\Honev[T]$ and the bound
  \begin{gather}
    \label{eq:135}
    \NHone[T]{\Op{R}_{T}\curl\Vu} \leq C \NLtwo{\curl \Vu}\;,
  \end{gather}
  where here and below no constant may depend on $\Vu$ or $p$. 
\end{enumerate}
Hence, as $\Vu\in\Hmv[T]{s}$, there exists $v\in\Hm[T]{1+s}$ such that
\begin{gather}
  \label{eqo:8}
  \Vu = \grad v + \Op{R}_{T}\curl\Vu\;.
\end{gather}
The continuity of $\Poinc_{T}$ reveals that, with a constant $C>0$ 
depending only on $T$,
\begin{gather}
\label{eqo:9}
  \SNHm[T]{v}{1+s} \leq \NHm[T]{\Vu}{s} + \SNHone[T]{\Op{R}_{T}\curl\Vu} \leq
  \NHm[T]{\Vu}{s} + C \NLtwo[T]{\curl\Vu} \;.
\end{gather}
By the assumptions of the lemma and \eqref{eq:72} we know that
\begin{gather}
  \label{eqo:73}
  \Poinc_{T} \curl\Vu \in \Cw^{1}_{p}(T)\;.
\end{gather}
By the commuting diagram property from Lemma~\ref{lem:prj1} and the projector
property of $\Pi_{T,p}^{1}$ the task is reduced to an interpolation estimate for
$\Pi_{T,p}^{0}$:
\begin{gather}
  \label{eqo:10}
  (\Op{Id}-\Pi_{T,p}^{1})\Vu \overset{\text{\eqref{eqo:8}}}{=}
  \grad (\Op{Id}-\Pi_{T,p}^{0})v + 
  \underbrace{(\Op{Id}-\Pi_{T,p}^{1})\Op{R}_{T}\curl\Vu}_{=0}
\end{gather}
As a consequence, invoking Thm.~\ref{thm:deb04},
\begin{multline}
  \label{eqo:11}
  \NLtwo[T]{(\Op{Id}-\Pi_{T,p}^{1})\Vu} \overset{\text{\eqref{eqo:10}}}{=}
  \SNHm[T]{(\Op{Id}-\Pi_{T,p}^{0})v}{1} \\ \leq
  C \frac{\log^{2/3}p}{p^{s}} \SNHm[T]{v}{1+s} \overset{\text{\eqref{eqo:9}}}{\leq} C
  \frac{\log^{3/2}p}{p^{s}}\bigl(\NHm[T]{\Vu}{s}+\NLtwo[T]{\curl\Vu}\bigr)\;,
\end{multline}
which furnishes a local version of the estimate. Squaring \eqref{eqo:11} and summing
over all tetrahedra finishes the proof.
\end{proof}

\section{Discrete compactness}
\label{sec:discrete-compactness}

Smoothness of the solenoidal part of the Helmholtz decomposition of $\Hcurl$ and
$\zbHcurl$, respectively, plays a pivotal role. It can be deduced from elliptic
lifting theorems for 2nd-order elliptic boundary value problems \cite[Ch.~6]{DAU88}.
Proofs of the following lemma can be found in \cite[Sect.~4.1]{HIP02} and
\cite[Sect.~3]{ABD96}. 

\begin{lemma}
  \label{lem:reg}
  For any Lipschitz polyhedron $\Omega\subset\bbR^{3}$ there is a $\frac{1}{2}<s\leq
  1$ such that $\Cx:=\zbHdiv\cap\Hcurl$ and $\Cx:=\Hdiv\cap\zbHcurl$ are continuously
  embedded into $\Hmv{s}$, that is,
  \begin{gather}
    \label{eq:97}
    \exists C=C(s,\Omega)>0:\quad
    \NHm{\Vu}{s} \leq C \left(\NHcurl{\Vu}+\NHdiv{\Vu}\right)\quad\forall \Vu \in \Cx\;.
  \end{gather}
\end{lemma}

We first verify the discrete compactness property of Def.~\ref{def:dc} for
$\epsilonbf\equiv 1$: consider a sequence $(\Vu_{p})_{p\in\mathbb{N}}$, which
satisfies
\begin{align}
  \label{eqo:12}
  \text{(i)}\; & \Vu_{p} \in \mathcal{W}_{p}^{1}(\mathcal{M})\;,\\
  \label{eqo:13}
  \text{(ii)}\;& \SPLtwo{\Vu_{p}}{\Vz_{p}} = 0 \quad
  \forall \Vz_{p}\in \{\Vv\in\mathcal{W}_{p}^{1}(\mesh):\;\curl\Vv=0\}\;,\\
  \label{eqo:17}
  \text{(iii)}\;& \NLtwo{\Vu_{p}} + \NLtwo{\curl\Vu_{p}} \leq 1\quad\forall p\in\mathbb{N}\;.
\end{align}

\begin{theorem}
  \label{thm:main}
  A sequence $(\Vu_{p})_{p\in\mathbb{N}}$ compliant with \eqref{eqo:12}--\eqref{eqo:17}
  possesses a subsequence that converges in $\Ltwov$.
\end{theorem}

\begin{proof}
  The proof resorts to the ``standard policy'' for tackling the problem of discrete
  compactness, introduced by Kikuchi \cite{KIK89,KIK01} for analyzing the $h$-version 
  of edge elements. It forms the core of most papers tackling the issue of 
  discrete compactness, see \cite[Thm.~2]{BDC03}, \cite[Thm.~11]{BCD06},
  \cite[Thm.~4.9]{HIP02}, \cite[Thm.~2]{DMS00}, \cite[Thm.~11]{BCD06}, etc.

  We start with the continuous Helmholtz decomposition of $\Vu_{p}$: let
  $\widetilde{\Vu}_{p}$ be the unique vector field in $\Hcurl$ with
\begin{align}
  \label{eq:98}
  \curl \widetilde{\Vu}_{p} & = \curl\Vu_{p}\;,\\
  \label{eq:99}
  \SPLtwo{\widetilde{\Vu}_{p}}{\Vz} & = 0 \quad
  \forall \Vz\in \Kern{\curl}\cap\Hcurl\;.
\end{align}
The inclusion $\grad\Hone\subset\Kern{\curl}$ enforces 
\begin{gather}
  \label{eqo:14}
  \Div\widetilde{\Vu}_{p} = 0\quad\text{in }\Omega
  \quad,\quad \widetilde{\Vu}_{p}\cdot\Bn = 0 \quad\text{on }
  \partial\Omega\;.
\end{gather}
Hence, by virtue of Lemma~\ref{lem:reg}, $\widetilde{\Vu}_{p}$ satisfies
\begin{gather}
  \label{eq:100}
  \widetilde{\Vu}_{p}\in \Hmv{s}\quad,\quad
  \NHm{\widetilde{\Vu}_{p}}{s} \leq C \NHcurl{\Vu_{p}}\;,
\end{gather}
where $C>0$ depends only on $\Omega$ and $\frac{1}{2}<s\leq 1$. 

In addition, $\widetilde{\Vu}_{p}$ is $\Ltwo$-orthogonal to $\Kern{\curl}\cap\Hcurl$,
see \eqref{eq:99}. Thus, using Ned\'el\'ec's trick \cite{NED80}, we have
\begin{gather}
  \label{eqo:15}
  \begin{aligned}
    \NLtwo{\widetilde{\Vu}_{p}-\Vu_{p}}^{2} & =
    \SPLtwo{\widetilde{\Vu}_{p}-\Vu_{p}}{{\widetilde{\Vu}_{p}-
      \Pi_{p}^{1}\widetilde{\Vu}_{p}+
      \Pi_{p}^{1}\widetilde{\Vu}_{p}-\Vu_{p}
    }} \\
  & = \SPLtwo{\widetilde{\Vu}_{p}-\Vu_{p}}{{\widetilde{\Vu}_{p}-
    \Pi_{p}^{1}\widetilde{\Vu}_{p}}}\;,
  \end{aligned}
\end{gather}
because, thanks to Lemma~\ref{lem:43} and \eqref{eq:98}, 
\begin{gather}
  \label{eq:115}
  \curl (\Pi_{p}^{1}\widetilde{\Vu}_{p}-\Vu_{p}) = 
  (\Pi_{p}^{2}-\mathsf{Id})\underbrace{\curl\Vu_{p}}_{\in \Cw^{2}_{p}(\mesh)} = 0
  \;.\\
  \label{eq:143}
  \Rightarrow\quad 
  \Pi_{p}^{1}\widetilde{\Vu}_{p}-\Vu_{p} \in 
  \{\Vv\in\mathcal{W}_{p}^{1}(\mesh):\;\curl\Vv=0\}\;.
\end{gather}
Hence, appealing to Lemma~\ref{lem:1}, with $C>0$ independent of $p$,
\begin{gather}
  \label{eqo:16}
  \NLtwo{\widetilde{\Vu}_{p}-\Vu_{p}} 
  \begin{aligned}[t]
    & \leq \NLtwo{\widetilde{\Vu}_{p}-\Pi_{p}^{1}\widetilde{\Vu}_{p}} 
    \leq C \frac{\log^{3/2}p}{p^{s}}
    \bigr(\SNHm{\widetilde{\Vu}_{p}}{s}+\NLtwo{\curl\widetilde{\Vu}_{p}}\bigr)\\
    & \leq
    C \frac{\log^{3/2}p}{p^{s}}\;\NHcurl{\Vu_{p}} \to 0 \quad\text{for }
  p\to\infty\;.
\end{aligned}
\end{gather}
Since bounded in $\Hmv{s}$, by Rellich's theorem ${(\widetilde{\Vu}_{p})}_{p\in\bbN}$
has a convergent subsequence in $\Ltwov$: owing to \eqref{eqo:16}, the same
subsequence of ${(\Vu_{p})}_{p\in\mathbb{N}}$ will converge in $\Ltwov$.
\end{proof}

\begin{theorem}
  \label{thm:maindir}
  Replacing $\Cw^{1}_{p}(\mesh)$ with ${{\Cw}}_{p}^{1}(\mesh)\cap\zbHcurl$ in
  \eqref{eqo:12}--\eqref{eqo:17}, the assertion of Thm.~\ref{thm:main} remains true.
\end{theorem}

\begin{proof}
  Since the projection based interpolation operators respect homogeneous Dirichlet
  boundary conditions, \textit{cf.} \eqref{eq:146}, 
  the proof runs parallel to that of Thm.~\ref{thm:main}.
  We point out that now $\widetilde{\Vu}_{p}\times\Bn=0$ on
  $\partial\Omega$ instead of $\widetilde{\Vu}_{p}\cdot\Bn = 0$, but 
  Lemma~\ref{lem:reg} can still be applied.
\end{proof}

Now we are able to switch from $\epsilonbf\equiv 1$ to general dielectric tensor, thus
completing the proof of the main theorem Thm.~\ref{thm:MAIN}.

\begin{proof}[of Thm.~\ref{thm:MAIN} in the Introduction] 
  We adapt the proof of \cite[Thm.~4.9]{HIP02}. Consider a $\Hcurl$-bounded sequence
  $\left(\Vw_{p}\right)_{p\in\bbN}$, $\Vw_{p}\in \Cw^{1}_{p}(\mesh)$, in the
  $L^{2}_{\epsilonbf}(\Omega)$-orthogonal complement $\Cx_{p}^{1}(\mesh)$ (see
  \eqref{eq:102}) of the discrete kernel of $\curl$, \textit{i.e.},
  \begin{gather}
    \label{eq:121a}
    \SPLtwo{\epsilonbf\Vw_{p}}{\Vz_{p}} = 0 \quad\forall \Vz_{p}\in\Kern{\curl}\cap
    \Cw^{1}_{p}(\mesh)\;.
  \end{gather}
  We continue with the  $\Ltwo$-orthogonal \emph{discrete} Helmholtz decomposition
  \begin{gather}
    \label{eq:122}
    \Vw_{p} = \Vu_{p} \oplus_{L^{2}} 
    \Vw_{p}^{0}\;,\quad
    \Vu_{p}\in\Cw^{1}_{p}(\mesh),\;\Vw_{p}^{0}\in\Kern{\curl}\cap \Cw^{1}_{p}(\mesh)\;.
  \end{gather}
  As $\SPLtwo{\Vu_{p}}{\Vz_{p}}=0$ for all $\Vz_{p}\in\Kern{\curl}\cap
  \Cw^{1}_{p}(\mesh)$, by Thm.~\ref{thm:main} we can find a subsequence, again
  denoted by $\left(\Vu_{p}\right)_{p\in\bbN}$, with
  \begin{gather}
    \label{eq:123}
    \begin{CD}
      \Vu_{p} @>{p\to\infty}>> \Vq\quad\text{in }\Ltwo\;.
    \end{CD}
  \end{gather}
  Since $\Vw_{p}$ satisfies \eqref{eq:121}, we conclude
  \begin{gather}
    \label{eq:124}
    \SPLtwo{\epsilonbf \Vw_{p}^{0}}{\Vz_{p}} = - \SPLtwo{\epsilonbf\Vu_{p}}{\Vz_{p}}
    \quad\forall \Vz_{p}\in \Kern{\curl}\cap\Cw^{1}_{p}(\mesh)\;.
  \end{gather}
  This can be regarded as a perturbed spectral Galerkin approximation of the
  following continuous variational problem: seek
  $\Vy\in\kHcurl:=\Kern{\curl}\cap\Hcurl$ such that
  \begin{gather}
    \label{eq:125}
    \SPLtwo{\epsilonbf \Vy}{\Vz} = - \SPLtwo{\epsilonbf\Vq}{\Vz}\quad\forall
    \Vz\in \kHcurl\;,
  \end{gather}
  which, obviously, has unique solution $\Vy$. From Strang's lemma
  \cite[Thm.~4.4.1]{CIA78} we infer
  \begin{gather}
    \label{eq:126}
    \NLtwo{\Vy-\Vw_{p}^{0}} \leq C \bigl(
      \inf\limits_{\Vz_{p}\in \Cw^{1}_{p}(\mesh)\cap\Kern{\curl}}
      \NLtwo{\Vy-\Vz_{p}} + \underbrace{\NLtwo{\Vu_{p}-\Vq}}_{\to 0\;\text{for }p\to\infty}
    \bigr)\;,
  \end{gather}
  where $C>0$ depends only on $\epsilonbf$. Next recall, that there is a
  representation
  \begin{gather}
    \label{eq:127}
    \Vy = \grad v + \Vh\;,\quad v\in \Hone\;,\quad
    \Vh\in\Chv^{1}(\Omega)\;,
  \end{gather}
  with $\Ch^{1}(\Omega)$ standing for the \emph{finite dimensional} first co-homology
  space $\Chv^{1}(\Omega)$ of harmonic vector fields, which is contained in
  $\kHcurl\cap\zbkHdiv$, \cite[Lemma~2.2]{HIP02} and \cite[Prop.~3.14,
  Prop.~3.18]{ABD96}. Owing to Lemma~\ref{lem:reg} it belongs to $\Hmv{s}$ for some
  $s>\frac{1}{2}$, which implies, thanks to Lemma~\ref{lem:1},
  \begin{gather}
    \label{eq:129}
    \NLtwo{\Vh-\Pi^{1}_{p}\Vh} \leq C \frac{\log^{3/2}p}{p^{s}}\SNHm{\Vh}{s}\to
    0\quad\text{for } p\to \infty\;.
  \end{gather}
  Further, the commuting diagram property of Lemma~\ref{lem:43} confirms that
  $\curl\Pi^{1}_{p}\Vh=0$. Besides, asymptotic density of the spectral family of
  Lagrangian finite element spaces in $\Hone$ means that also the first term on the
  right hand side of \eqref{eq:126} tends to zero as $p\to\infty$.

  Thus, selecting the same subsequence of $\left(\Vw_{p}\right)_{p\in\bbN}$
  (and keeping the notation), it is immediate that
  \begin{gather}
    \label{eq:128}
    \begin{CD}
      \Vw_{p} @>{p\to\infty}>> \Vq + \Vy\quad\text{in }\Ltwo\;.
    \end{CD}
  \end{gather}
  The case of $\Vw_{p}\in\Cw^{1}_{p}(\mesh)\cap\zbHcurl$ is amenable to almost the
  same proof: the boundary conditions are imposed on all fields and in the
  counterpart of \eqref{eq:127} the second co-homology space $\Chv^{2}(\Omega)$ has
  to be considered \cite[Lemma~2.2]{HIP02}.
\end{proof}

\section{Acknowledgment}
\label{sec:acknowledgment}

The author would like to thank his colleague C. Schwab for pointing out the
crucial reference \cite{CMI08}. He is grateful to A. Buffa, M. Costabel,
M. Dauge, and L. Demkowicz, and all the other mathematicians who have laid
the foundations for this work.

%\bibliographystyle{siam}
%\bibliography{lit}

\end{document}